\documentclass[11pt]{amsart}

\usepackage{graphicx}
\usepackage{color}
\usepackage{amssymb}
\usepackage{enumitem}
\usepackage{hyperref}
\newcommand{\unYoel}[1]{}

\newcommand{\R}{\mathbb{R}}
\newcommand{\C}{\mathbb{C}}

\newcommand{\vectornorm}[1]{\left\|#1\right\|}

\newcommand{\restrict[3]}{\left. {#2}\right|_{#3}}

\DeclareMathOperator{\inj}{inj}

\setlist[enumerate,1]{label={(\alph*)}}
\setlist[enumerate,2]{label={(\roman*)}}

\newtheorem{tm}{Theorem}[section]

\newtheorem{lm}[tm]{Lemma}
\newtheorem{cy}[tm]{Corollary}

\theoremstyle{definition}
\newtheorem{df}[tm]{Definition}

\newtheorem{as}{Assumption}[tm]

\theoremstyle{remark}
\newtheorem{rem}[tm]{Remark}

\title{A thick-thin decomposition of $J$-holomorphic curves}
\author{Yoel Groman}

\begin{document}

\maketitle
\tableofcontents
\begin{abstract}
We show the existence of a thick thin decomposition of the domain of a pseudo holomorphic curve with boundary. The geometry of the thick part is bounded uniformly in the energy. Furthermore, in the thick part, there is a uniform bound on the differential which is exponential in the energy. The thin part consists of annuli of small energy the number of which is at most linear in the energy and genus. The decomposition can be seen as a quantitative version of Gromov compactness which applies before passing to the limit.
\end{abstract}
\section{Introduction}

A basic tool in the study of the moduli space $\mathcal{M}_g$ of compact Riemann surfaces of genus $g\geq 2$ is the thick thin decomposition of hyperbolic structures. Namely, for $\Sigma\in\mathcal{M}_g,$ let $h$ be the unique conformal metric $h$ of constant curvature $-1$. For $x\in\Sigma$ denote by $\inj(\Sigma,x;h)$ the radius of injectivity of $(\Sigma,h)$ at $x.$ Write
\begin{gather}
Thick(\Sigma;h):=\{x\in\Sigma|\inj(\Sigma,x;h)\geq\sinh^{-1}(1)\}\notag\\
Thin(\Sigma;h):=\{x\in\Sigma|\inj(\Sigma,x;h)<\sinh^{-1}(1)\}\notag.
\end{gather}
Then $Thin$ consists of at most $3g-3$ disjoint cylinders and the components of $Thick$ have geometry that is bounded uniformly in $g$. Among other things, the thick thin decomposition provides an intuitive picture of the Deligne Mumford compactification of $\mathcal{M}_g.$

This paper is concerned with an analogous construction for the moduli spaces $\mathcal{M}_g(M,J;A)$ of $J$-holomorphic curves of genus $g$ in a symplectic manifold $(M,\omega)$ representing $A\in H_2(M;\mathbb{Z})$ with $J$ an $\omega$-tame almost complex structure $J$. Namely, for
\[
(u:\Sigma\to M)\in \mathcal{M}_g(M,J;A),
\]
we construct a decomposition
\[
\Sigma=Thick(\Sigma;u)\cup Thin(\Sigma;u).
\]
$Thin(\Sigma;u)$ consists of  disjoint annuli and cylinders whose number is proportional to $g+\int_\Sigma u^*\omega.$ With respect to the standard cylindrical metric on $Thin(\Sigma;u)$, $|du|$ decays exponentially in the distance from $\partial Thin(\Sigma;u).$ The components of $Thick(\Sigma;u)$, once properly normalized, have uniformly bounded geometry with the bounds exponential in the energy of $u$. Furthermore, on $Thick$ there is a bound on $|du|$ which is exponential in the energy. We construct an analogous decomposition for bordered $J$-holomorphic curves with boundary in a Lagrangian submanifold $L.$ This time, it is the complex double of the domain which is decomposed. Our thick thin decomposition  is related to Gromov compactness in the same way the hyperbolic thick thin decomposition of Riemann surfaces is related to the Deligne-Mumford compactification.
\subsection{The main result}
To formulate the result more precisely, we introduce the following definitions.
\begin{df}
Let $\Sigma$ be a closed Riemann surface and let $h$ be a conformal metric of constant curvature on $\Sigma$. A \textbf{geodesic annulus} in $\Sigma$ is a doubly connected subset of the form
\[
A(r_1,r_2,p;h)=\{y\in\Sigma|r_1<d_h(y,p) <r_2\},
\]
for some $p\in\Sigma$, and $0<r_1<r_2< \inj(p;h)$. For a simple closed geodesic $\gamma$ in $\Sigma,$ let $R_{\gamma}$ be the width of a geodesic tubular neighborhood of $\gamma$. Suppose $\gamma$ is oriented with unit normal $v$. A \textbf{geodesic cylinder} in $\Sigma$ is a doubly connected subset of the form
\[
C(r_1,r_2,\gamma;h)=\{y=\exp rv_p|p\in\gamma,r\in(r_1,r_2)\},
\]
for some
\[
-R_{\gamma}\leq r_1<r_2\leq R_{\gamma}.
\]
A \textbf{bubble decomposition} of $\Sigma$ is a collection of geodesic annuli and geodesic cylinders in $\Sigma$ with pairwise disjoint closures. Write
\begin{align}
Thin(\mathcal{B})&:=\bigcup_{I\in \mathcal{B}}I,\notag\\
Thick(\mathcal{B})&:=\Sigma\setminus Thin(\mathcal{B})\notag.
\end{align}

\begin{figure}[h]
\centering
\includegraphics[scale=0.6]{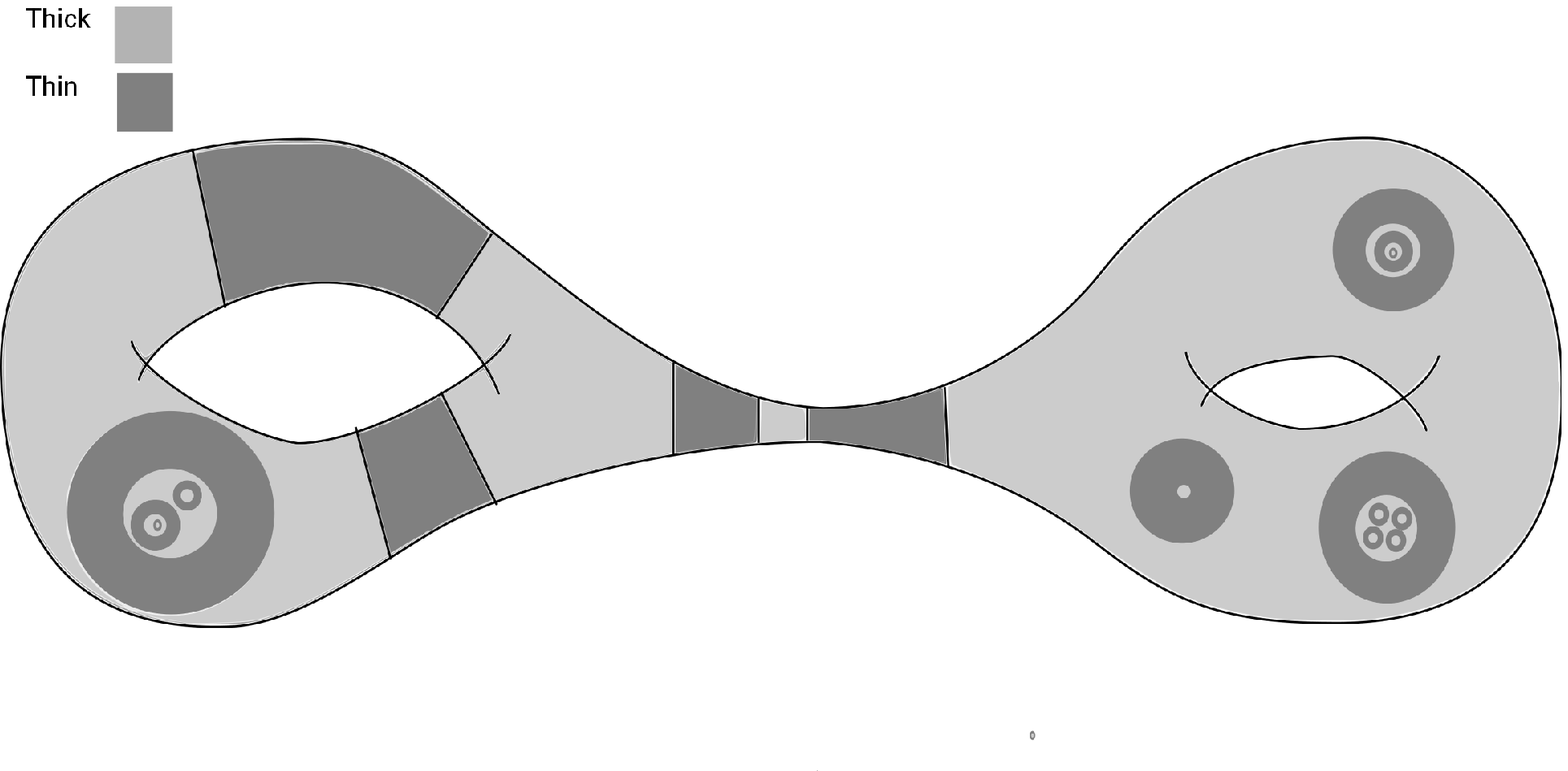}
\caption{}\label{Fig1}
\end{figure}

For a bubble decomposition $\mathcal{B}$, let $V_{\mathcal{B}}$ denote a finite set with a bijection
\[
V_{\mathcal{B}}\to \pi_0(Thick(\mathcal{B})),\qquad v\mapsto \Sigma_v.
\]
Here, $\pi_0(\cdot)$ denoting the set of connected components.
\end{df}

\begin{df}\label{dfBubDecEst}
Let $\Sigma$ be a closed Riemann surface, let $h$ be a conformal metric of constant curvature on $\Sigma$.  Denote by $\nu_h$ the volume form on $\Sigma$. Let $\mu$ be a measure on $\Sigma$ which is absolutely continuous with respect to any smooth volume form on $\Sigma$, and denote by
\[
\frac{d\mu}{d\nu_h}
\]
the Radon-Nikodym derivative of $\mu$ with respect to $\nu_h.$ Let $a,b,\delta>0.$ A \textbf{$(\mu,h)$-adapted bubble decomposition} $\mathcal{B}$ with constants $a,b,\delta,$ is a bubble decomposition satisfying the following estimates.
\begin{enumerate}
\item \label{dfBubDecEst11}\textbf{Exponential decay in the thin part.} For any $I\in\mathcal{B},$ denote by $Mod(I)$ the modulus of $I$ and by $h_{st}$ the unique conformal metric such that $(I,h_{st})$ is isometric to $(0, Mod(I))\times S^1 $. Then for any $p\in I$,
\begin{equation}\label{eqExpDec}
\frac{d\mu}{d\nu_{h_{st}}}(p)\leq ae^{-bd_{h_{st}}(p,\partial I)}.
\end{equation}
\item\label{dfBubDecEst13}\textbf{Bounded geometry and derivative.} For any $v\in  V_{\mathcal{B}}$ let $g_v$ be the genus of $\Sigma_v$ and $d_v$ its diameter with respect to $h$. Let $n_v:=|\pi_0(\partial\Sigma_v)|$ and $\mu_v:=\mu(\Sigma_v)$. Let
    \[
    s_v:=\frac{2(g_v+1)}{d_v}
    \]
   and define $h_v:=s_v^2h|_{\Sigma_v}.$ Then the following hold.
\begin{enumerate}
\item
\begin{equation}\label{eqThickDerEst}
\sup_{p\in \Sigma_v}\frac{d\mu}{d\nu_{h_v}}(p)\leq a e^{b(\mu_v+n_v)}.
\end{equation}
\item\label{boundgeo2}
\[
\inf_{p\in \Sigma_v}\inj(\Sigma_v,p;h_v)\geq ae^{-b(\mu_v+n_v)}\footnote{See Remark~\ref{remInjBoundary} below for the definition of $\inj$ for surfaces with boundary.}.
\]
\item For any component $\gamma$ of $\partial\Sigma_v$,
\[
\ell(\gamma;h_v)>ae^{-b(\mu_v+n_v)}.
\]
\item\label{boundgeo4}
For any two distinct components $\gamma_1$ and $\gamma_2$ of $\partial\Sigma_v$,
\[
d_{h_v}(\gamma_1,\gamma_2)\geq ae^{-b(\mu_v+n_v)}
\]
\end{enumerate}
\item\label{dfBubDecEst12}\textbf{Stability.} For any $v\in V_{\mathcal{B}}$ we have either
\[
\mu_v\geq \delta,
\]
 or
\[
2genus(\Sigma_v)+|\pi_0(\partial\Sigma_v)|\geq 3.
\]
\end{enumerate}

\end{df}
\begin{rem}
Note that because of the restriction to constant curvature metrics, only in the genus 0 case does the property of $(\mu,h)$-adaptedness depend on $h$. In the other cases it would be more proper to talk of $\mu$-adaptedness.
\end{rem}

\begin{rem}\label{rmfin}
Note that the stability condition implies
\[
|\pi_0(Thick(\mathcal{B}))|\leq 2g+2\frac{\mu(\Sigma_v)}{\delta}-3,
\]
and a similar estimate for $|\pi_0(Thin(\mathcal{B}))|.$
\end{rem}

\begin{rem}\label{rmcpt}
Fix an $E>0$ and a $g\in\mathbb{N}$. The bounds~\ref{boundgeo2}-\ref{boundgeo4} of Definition~\ref{dfBubDecEst}\ref{dfBubDecEst13} imply that in the space of Riemannian manifolds with boundary equipped with the $C^{\infty}$ topology, there exists a compact subset $K=K(g,E,a,b,\delta)$ with the following significance. For all measured Riemann surfaces $(\Sigma,\mu)$ with $genus(\Sigma)\leq g$ and $\mu(\Sigma)\leq E,$ any constant curvature metric $h$ on $\Sigma,$ and any $(\mu,h)$-adapted bubble decomposition $\mathcal{B}$ of $\Sigma$, the components of $Thick(\mathcal{B})$ belong to $K.$ This follows from Theorem 3.3.1 in \cite{AMK}.
\end{rem}

We now state the main result. Let $(M,\omega)$ be a compact symplectic manifold and $J$ an $\omega$-tame almost complex structure. For a Riemann surface $\Sigma$ and a $J$-holomorphic curve
\[
u:\Sigma\to M,
\]
and for any subset $U\subset\Sigma$, write
\[
\mu_u(U):=\int_Uu^*\omega.
\]

\begin{tm}\label{tmBubDecEst1}
Let $M$ be compact. Let $\mathcal{F}$ be the family of closed non-constant J-holomorphic curves in $M$. Then for every $(\Sigma,u)\in \mathcal{F}$ there is a conformal metric $h$ of constant curvature on $\Sigma$ and a $(\mu_u,h)$-adapted bubble decomposition $\mathcal{B}_u$ of $\Sigma$ with constants depending on $\mathcal{F}$ only.
\end{tm}

\begin{rem}\label{remSphChoice}
If we were to allow an arbitrary constant curvature metric in the genus 0 case, a simple counterexample to the theorem could be obtained as follows. Let $h$ be the standard metric on $S^2 =\C\cup\{\infty\}$, let $u:S^2\to M$ be a non-constant $J$-holomorphic curve, let $\psi_n:S^2\to S^2$ be given by $\psi(z)=nz$ for any $z\in\C\subset S^2$ and let $u_n= u_n\circ\psi_n$. Then there are no uniformly $(\mu_u,h)$-adapted bubble decompositions for this sequence.
\end{rem}

\subsection{Curves with boundary}

\begin{df}\label{dfCplxDbl}
 For any Riemann surface $\Sigma=(\Sigma,j)$, write $\overline{\Sigma}:=(\Sigma,-j)$. The \textbf{complex double} is the Riemann surface
 \[
 \Sigma_{\C} :=\Sigma\cup\overline{\Sigma},
  \]
 where the surfaces are glued together along the boundary by the identity. The complex structure on $\Sigma_{\C}$ is the unique one which coincides with $j$ and with $-j$ when restricted suitably. $\Sigma_{\C}$ is endowed with a natural anti-holomorphic involution and for any $z\in\Sigma_{\C}$ we denote by $\overline{z}$ the image of $z$ under this involution.
\end{df}

\begin{df}
Let $\Sigma$ be a connected Riemann surface. A subset $S\subset\Sigma_{\C}$ is said to be \textbf{clean} if either $S=\overline{S}$ or $S\cap \overline{S}=\emptyset$.
\end{df}

\begin{df}
A bubble decomposition of $\Sigma_\C$ is said to be conjugation invariant if and only if all $I\in\mathcal{B}$ are clean and
\[
I\in \mathcal{B}\qquad\Rightarrow\qquad \overline{I}\in\mathcal{B}.
\]
\end{df}

\begin{tm}\label{tmBubDecEst2}
Let $\mathcal{F}$ be the family of non-constant $J$-holomorphic curves in $M$ with boundary in a compact Lagrangian submanifold $L$. Then for every $(\Sigma,u)\in \mathcal{F}$ there is a conjugation invariant conformal constant curvature metric $h$ on $\Sigma_\C$ and a conjugation invariant $(\mu_u,h)$-adapted bubble decomposition $\mathcal{B}_u$ of $\Sigma_\C$ with constants depending on $\mathcal{F}$ only.
\end{tm}
\subsection{The non-compact setting}
\begin{df}
For any Riemannian manifold $X$ with sub-manifold $Y$ and $\epsilon>0$, we say that $Y$ is $\epsilon$-Lipschitz if
\[
\frac{d_X(x,y)}{\min\{1,d_Y(x,y)\}} \geq\epsilon\qquad \forall x\neq y\in Y.
\]
We say that $Y$ is Lipschitz if there is an $\epsilon$ such that $Y$ is $\epsilon$-Lipschitz.
\end{df}
Denote by $g_J$ the symmetrization of the positive definite form $\omega(\cdot,J\cdot).$ Denote by $R$ the curvature of $g_J$, by $B$ the second fundamental form of $L$ with respect to $g_J$ and for any tensor $T$ on $M$ or $L$ let $\|T\|_n$ denote the $C^n$ norm of $T$ with respect to $g_J$.
\begin{df}\label{BoundCond}
 Let $S$ be a family of compact Riemann surfaces, possibly with boundary. We say that the data of $S$ together with $(M,\omega,L,J)$ comprise a \textbf{bounded setting} if $M$ and $L$ are complete with respect to $g_J$ and one of the following holds.
\begin{enumerate}
\item \label{BoundCond2}$L=\emptyset$ and
\[
\max\left\{\vectornorm{R},\vectornorm{J}_2,\frac1{\inj(M;g_J)}\right\}<\infty.
\]
\item\label{BoundCond3}
$L$ is Lipschitz and
\[
\max\left\{\vectornorm{R}_2,\vectornorm{J}_2,\vectornorm{B}_2,\frac1{\inj(M;g_J)}\right\}<\infty.
\]
\item\label{BoundCond4}
Each connected component $L'$ of $L$ is Lipschitz and
\[
\max\left\{\vectornorm{R}_2,\vectornorm{J}_2,\vectornorm{B}_2,\frac1{\inj(M;g_J)}\right\}<\infty.
\]
Furthermore, there is an $\epsilon>0$ such that for each $(u,\Sigma)\in\mathcal{F}$, there is a conformal metric $h$ of constant curvature $0,\pm 1,$ of unit area in case of zero curvature, such that $\partial\Sigma$ is totally geodesic and $\epsilon$-Lipschitz.

\end{enumerate}
\end{df}

\begin{tm}\label{tmBubDecEst3}
Let $\mathcal{F}$ be the family of non-constant $J$-holomorphic curves in $M$ with boundary in $L$ and domain in a set $S$ of Riemann surfaces such that $S$ and $(M,\omega,J,L)$ comprise a bounded setting. Then for every $(\Sigma,u)\in \mathcal{F},$ there is a conjugation invariant conformal constant curvature metric $h$ on $\Sigma_\C$ and a conjugation invariant $(\mu_u,h)$-adapted bubble decomposition $\mathcal{B}_u$ of $\Sigma_\C$ with constants depending on $\mathcal{F}$ only.
\end{tm}

\subsection{Relation to Gromov compactness}
 Fix an $E>0$ and a $g\in\mathbb{N}$. Then for all $u$ such that $genus(\Sigma)\leq g$ and $\mu_u(\Sigma)\leq E$, the components of $Thick(\mathcal{B}_u)$ are elements of $K$, where $K$ is as in Remark~\ref{rmcpt}. Furthermore, by Remark~\ref{rmfin}, $|\pi_0(Thick(\mathcal{B}_u)|$ is bounded uniformly in the set of all such $u$. By conformality we have that $|du|^2_h=\frac{d\mu_u}{d\nu_h}$.  Together with estimate \eqref{eqThickDerEst} and elliptic regularity, we obtain $C^{\infty}$ compactness of the restriction of $J$-holomorphic curves to their thick parts.

To see what happens in the thin part, let us elaborate on the geometric meaning of Definition~\ref{dfBubDecEst}\ref{dfBubDecEst11}. Let $I$ be an open cylinder, let $u:I\to M$ be $J$-holomorphic and Let
\[
\psi:I_L:=(-L,L)\times S^1\rightarrow I
\]
be a biholomorphism. Let
\[
h_{cone}=\sqrt{a}e^{-\frac{b}{2}(L-|r|)}h_{st}.
\]
Then for $r\neq 0,$ $h_{cone}$ is a conformal metric on $I_L$ whose shape is as an approximate cone as in the left of Figure~\ref{FigCone}. By inequality~\eqref{eqExpDec},
\begin{align}\label{EqConederEst}
\frac{d\mu_{u\circ\psi}}{d\nu_{h_{cone}}}=\frac{d\mu_{u\circ\psi}}{d\nu_{h_{st}}}\frac{d\nu_{h_{st}}}{d\nu_{h_{cone}}}\leq 1.
\end{align}
As $L\rightarrow\infty$ the approximate cones converge to an actual cone. See Figure~\ref{FigCone}.
\begin{figure}[h]
\centering
\includegraphics[scale=0.5]{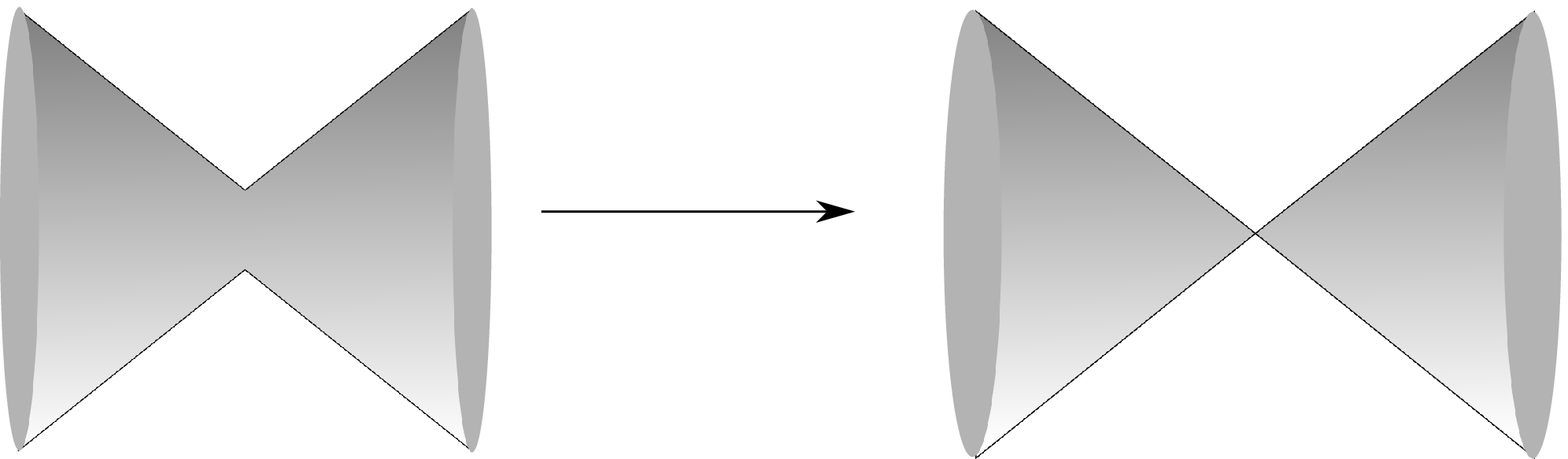}
\caption{}\label{FigCone}
\end{figure}
Gromov's compactness theorem is a consequence of this discussion, of elliptic regularity and of removal of singularities. Use of convergence theory of Riemannian manifolds in the context of Gromov compactness appears also in \cite{Pa} and \cite{SiTi}.

\subsection{The thin part}
The specification of the thin part of a $J$-holomorphic curve $u:\Sigma\to M$ is more involved then that of a hyperbolic surface. Furthermore, as a subset of $\Sigma$ it appears to involve some choices which have to be made for each $u$. However, the combinatorial structure of the thick thin decomposition, e.g. the number of components of $Thin,$ is independent of any such choices. For simplicity we describe the thin part of a closed $J$-holomorphic curve $u:\Sigma\to M$.

We recall the cylinder inequality \cite[Lemma 4.7.3]{MS2}. Let $I_a:=[-a,a]\times S^1$. The cylinder inequality states that there are constants $\delta$ and $c$ such that for any $J$-holomorphic map $u:I_a\to M$ we have
\begin{equation}\label{IntCylIneq}
\mu_u(I_a)\leq\delta\qquad\Rightarrow\qquad \mu_u(I_{a-t})\leq e^{-ct}\mu_u(I_a),
\end{equation}
for $t\in[\log 2,a]$.

\begin{df}\label{dfPrelimLN}
An $L$-long neck is a geodesic cylinder or annulus $I\subset\Sigma$ such that $Mod(I)\geq 4L$, $\mu_u(I)\leq\delta/6$ and each component $A$ of  $\Sigma\setminus I$ is stable in the sense that one of the following conditions holds:
\begin{enumerate}
\item
$\mu(A)\geq\delta$;
\item
$2\,genus(A) + |\pi_0(\partial A)|\geq 3.$
\end{enumerate}

\end{df}

For $L$ large enough we define an equivalence relation on the set $LN$ of $L$-long necks as follows. Suppose $I_1,I_2\in LN$. Then $I_1\sim I_2$ if and only if there exists an annulus $I$, not necessarily geodesic, such that $I_1$ and $I_2$ are nontrivially embedded in $I$ and $\mu(I)\leq \delta/2.$ That $\sim$ is indeed equivalence relation for $L$ large enough follows from the cylinder inequality and the stability condition. See Lemma~\ref{lmSimEq} below. Furthermore, each equivalence class is shown to contain an element of maximal modulus.

Pick an element $A_c$ of maximal modulus from each $\sim$-equivalence class $c$ and let $L_c=\frac12Mod(A_c)$. There is a biholomorphism
\[
f_c:A_c\to I_{\frac12Mod(A_c)}=[-L_c,L_c]\times S^1,
\]
unique up to automorphisms of the cylinder. The components of the thin part are the annuli $f_{c}^{-1}(I_{L_{c}-L}).$ These are shown in the text to be disjoint for $L$ large enough but chosen independently of the curve. There does not appear to be a unique maximal element in each equivalence class. Hence the choices referred to at the beginning of the subsection.

\subsection{Idea of the proof}
Let $\mathcal{B}$ be the set of components of the thin part as outlined in the previous subsection. In the text we show that $Thick(\mathcal{B})$ contains no long necks. It turns out that when $genus(\Sigma)>0,$ this implies that $\mathcal{B}$ is $(\mu_u,h)$-adapted. Let us sketch for example how to obtain the derivative estimate in $Thick(\mathcal{B}).$

For this, recall the gradient inequality \cite[Lemma 4.3.1]{MS2} which says that there is a constant $\delta'>\delta$ such that for any ball $B_r(p)\subset\Sigma$ we have
\[
\mu_u(B_r(p))\leq \delta'\qquad\Rightarrow\qquad \frac{d\mu_u}{d\nu_h}(p)\leq\frac{1}{r^2}\mu_u(B_r(p)).
\]

Let $v\in V_{\mathcal{B}}.$ Suppose for concreteness that $\Sigma_v$ is a geodesic disk $D=B_1(z;h_v)\subset\Sigma$. In this paragraph all quantities are measured with respect to $h_v$, so we omit it from the notation.  Let $p\in\Sigma_v$ be a point where the derivative obtains its maximum. Using the gradient inequality and the construction of $\mathcal{B}$ there is an a priori bound on the derivative in the annulus $B_{1}(z)\setminus B_{1/2}(z)$. Suppose $p\in B_{1/2}(z)$ and let $d=\frac{d\mu_u}{d\nu_{h_v}}(p)$. Suppose $d>4$ and consider the annulus $A=B_{1/2}(p)\setminus B_{1/d}(p)$. Then $\Sigma\setminus A$ is stable in the sense appearing in Definition~\ref{dfPrelimLN}. Indeed,  the gradient inequality implies
\begin{align}\label{DefDenseDisc}
\mu_u(B_{1/d}(p))\geq\delta'.
\end{align}
The component $\Sigma\setminus B_{1/2}(p)$ is clearly stable by the assumption on the genus. Since $D$ is free of long necks, for any $\frac1{d}<r_1<r_2<\frac12$ such that
\[
\log{r_2/r_1}>L,
\]
we must have
\[
\mu_u\left(B_{r_2}(p)\setminus B_{r_1}(p)\right)>\delta/6.
\]
In particular,
\[
\log d\leq \frac{6L}{\delta}\mu_u(A)\leq \frac{6L}{\delta}\mu_u(\Sigma_v),
\]
which is just inequality~\eqref{eqThickDerEst}.
\subsection{Acknowledgements}
The author would like to thank his PhD advisor J. Solomon for countless valuable comments and suggestions and for helpful criticism. The author is grateful to the Azrieli foundation for the award of an Azrieli fellowship. The author was partially supported by ERC Starting Grant 337560.

\section{Preliminaries}\label{SecConfGeom}

\subsection{Annuli}
\begin{df}\label{DfModAx}
A \textbf{standard annulus} $I$ is a surface of the form $K\times S^1$ with $K\subset\R$ an interval which may be open, closed or half closed. We denote by $h_{st}$ the product metric on $I$ which assigns to $S^1$ the length $2\pi$. We let $j_{st}$ be the complex structure induced on $I$ by $h_{st}$ and the product orientation on $I$. We take $Mod(I):=|K|$, where $|\cdot|$ denote the Lebesgue measure. An \textbf{Annulus} $(I,j)$ is a doubly connected surface with complex structure $j$. Up to translation it is bi-holomorphic to a unique standard annulus $I_{st}$. We define $Mod(I,j):=Mod(I_{st})$. When the complex structure is clear from the context we omit it.

Let $(I,j)$ be an annulus and let $h$ be a conformal Riemannian metric on $I$. We call global cylindrical coordinates $(\rho,\theta)$ on $I$, with
\[
a \leq \rho \leq b, \qquad 0 \leq \theta < 2\pi,
\]
\textbf{axially symmetric} if
\begin{equation}\label{eq:axsy}
h =d\rho^2+h_{\theta}(\rho)^2d\theta^2.
\end{equation}
We say $h$ is \textbf{axially symmetric} if $I$ has axially symmetric coordinates.
In this case, the conformal length of $I$ is given by
\begin{equation}\label{ModEq}
Mod(I,j)=\int_a^b\frac1{h_{\theta}(\rho)}d\rho.
\end{equation}
\end{df}

\begin{df}\label{dfSubCyinlder}
Let $I$ be an annulus and let $L=Mod(I)$. Suppose $L<\infty$. Then there is a biholomorphism $f:K\times S^1\rightarrow I$ with $K$ an interval whose infimum is the origin. The map $f$ is unique up to a rotation and a holomorphic reflection. A \textbf{sub-cylinder} of $I$ is a subset of the form
\[
f(K'\times S^1),
\]
with $K'\subset K$ an interval.  For $a\leq b\in K$ we write
\[
S(a,b;I):=f([a,b]\times S^1)\subset I.
\]
We also define
\[
C(a,b;I):=S(a,L-b;I),
\]
for $a,b$ in the appropriate range. Note that composing $f$ with a holomorphic reflection of $K\times S^1$ replaces $S(a,b)$ with $S(L-b,L-a)$. When applying the above notations we shall be careful to remove this ambiguity. On the other hand, the notation $C(a,a;I)$ is well defined. Denote by $K^c$ the closure of $K$ and by $I^c$ the closure of $I$. It is convenient to extend the above definitions to $a,b\in K^c$ by defining $S(a,b;I):=S(a,b;I^c)\cap I$ and $C(a,b;I^c):=C(a,b;I)\cap I^c$.
\end{df}

\begin{df}\label{DfCofRad}
Let $U$ be a Riemann surface biholomorphic to the unit disk $D_1$. Let $h$ be a conformal metric on $U$ and let $z\in U$. Then there is a biholomorphism $\phi:U\rightarrow D_1$ with $\phi(z)=0,$ unique up to rotation. The \textbf{conformal radius of $U$ viewed from $z$} is defined to be
\[
r_{conf}(U,z;h):=1/\|d\phi(z)\|_h.
\]

\end{df}

Note that $r_{conf}(U,z;h)$ is not conformally invariant, since it depends on the metric at $z$. However, let $\nu_h$ denote the volume form of $h,$ let $\mu$ be an absolutely continuous measure on $U$. Then the expression $\frac{d\mu(z)}{d\nu_{h}}r_{conf}^2(z)$ is conformally invariant.

The cases of interest for us will be conformal radii of geodesic disks with metrics of constant curvature $K,$ viewed from their center.
In these cases, the metric can be written in polar coordinates as
\begin{equation}\label{eq:pch}
h = d\rho^2 + h_\theta^2(\rho) d\theta^2,
\end{equation}
where
\begin{equation}\label{eq:forhth}
h_\theta(\rho) = \begin{cases}
\sinh(\rho), & K = -1, \\
\rho, & K = 0, \\
\sin(\rho), & K = 1.
\end{cases}
\end{equation}
So, the conformal radius of $B_r(p)$ viewed from $p$ is given by
\[
r_{conf} = \exp(f(r))
\]
where $f$ is the function defined by
\[
f'(r) = \frac{1}{h_\theta(r)}, \qquad f(r) = \log(r) + O(r) \text{  as $r \to 0$}.
\]
More explicitly,
\begin{equation}\label{eq:f}
f(r) = \log(r) + \int_0^r \left(\frac{1}{h_{\theta}(\rho)} - \frac{1}{\rho} \right)d\rho.
\end{equation}
It follows from equation~\eqref{eq:f} that
\begin{equation}\label{eq:bdrc1}
r_{conf} \geq r, \qquad K = 0,1,
\end{equation}
and for any $\kappa$ there exists a constant $c > 0$ such that
\begin{equation}\label{eq:bdrc2}
r_{conf} \geq c r, \qquad K = -1, \; r < \kappa.
\end{equation}

\subsection{Collars}
For later reference we include a statement of the thick thin decomposition for surfaces of genus $g>1$. In the following we assume the surfaces are endowed with their unique metric $h$ of constant curvature $-1.$

\begin{tm}\label{TmThTh}\cite[4.1.1]{Bu}
Let $\Sigma$ be a compact Riemann surface of genus $g\geq2$, and let $\gamma_1,...,\gamma_m$ be pairwise disjoint simple closed geodesics on $\Sigma$. Then the following hold:
\begin{enumerate}
\item
$m\leq 3g-3$.
\item
There exist simple closed geodesics $\gamma_{m+1},...,\gamma_{3g-3},$ which, together 	 with $\gamma_1,...,\gamma_m$, decompose $\Sigma$ into pairs of pants.
\item\label{it:coll}
The collars
\begin{equation}
\mathcal{C}(\gamma_i)=\{p\in \Sigma|dist(p,\gamma_i)\leq w(\gamma_i)\}\notag
\end{equation}
of widths
\begin{equation}		 w(\gamma_i)=\sinh^{-1}\left(1/\sinh\left(\frac1{2}\ell(\gamma_i)\right)\right)\notag
\end{equation}
are pairwise disjoint for $i=1,...,3g-3$.
\item\label{it:spco}
Each $\mathcal{C}(\gamma_i)$ is isometric to the cylinder $[-w(\gamma_i),w(\gamma_i)]\times S^1$ with the Riemannian metric
\[
d\rho^2+\frac{\ell^2(\gamma_i)\cosh^2(\rho)}{4\pi^2} d\theta^2.
\]
\end{enumerate}
\end{tm}
Denote by $\inj(p;h)$ the radius of injectivity of $\Sigma$ at $p\in \Sigma$, i.e. the supremum of all $r$ such that $B_r(p)$ is an embedded disk.
\begin{tm}\label{TmThTh2}\cite[4.1.6]{Bu}
	Let $\beta_1,...,\beta_k$ be the set of all simple closed geodesics of length $\leq \sinh^{-1}1$ on $\Sigma$. Then $k\leq3g-3$ and the following hold.
	\begin{enumerate}
		\item The geodesics $\beta_1,...,\beta_k$ are pairwise disjoint.
		\item \label{it:lbd} $\inj(p;h) > \sinh^{-1}1$ for all $p\in \Sigma-(\mathcal{C}(\beta_1)\cup...\cup\mathcal{C}(\beta_k))$.
		\item \label{it:fml} If $p\in\mathcal{C}(\beta_i)$,and $d=dist(p,\partial \mathcal{C}(\beta_i))$, then
\begin{align}\label{injEq}
			\sinh (\inj(p;h)) =\cosh\frac1{2}\ell(\beta_i)\cosh d-\sinh d.
		\end{align}
	\end{enumerate}
\end{tm}

\begin{df}\label{dfCNeighbGeo}
Let $\Sigma$ be a closed Riemann surface. Let $h$ be a metric of constant curvature $0,\pm 1$ on $\Sigma$. Let $\gamma\subset \Sigma$ be a simple closed geodesic in $\Sigma$. In Theorem \ref{TmThTh}, $\mathcal{C}(\gamma)$ was defined when $genus(\Sigma)>1$. We extend the definition to the case $genus(\Sigma)\leq1$ by letting
\[
\rho_{max}=\max_{\{p\in\Sigma\}}d(p,\gamma),
\]
and
\[
\mathcal{C}(\gamma)=\{p\in\Sigma|d(p,\gamma)<\rho_{max}\}.
\]

When $genus(\Sigma)=0$, this is the sphere with two antipodes removed. It is also easy to verify that when $genus(\Sigma)=1$, this is a torus with a geodesic parallel to $\gamma$ removed. Global cylindrical coordinates $\rho$ and $\theta$ are defined on $\mathcal{C}(\gamma)$ in the same way as for $genus(\Sigma)>1$. Namely, $\rho(p)=d(p,\gamma;h)$ for any $p\in\mathcal{C}(\gamma)$ and $\theta$ maps lines of constant $\rho$ to $S^1$ isometrically up to multiplication with an overall constant.
\end{df}
\subsection{Cleanness}
\begin{lm}\label{lmCleanUnion}
Let $\Sigma$ be a Riemann surface and let $I_1$ and $I_2$ be clean subsets of $\Sigma_{\C}$. Then $I_1\cup I_2$ is clean if and only if at least one of the following holds:
\begin{enumerate}
\item\label{lmCleanUnioncond1}
$I_1$ and $I_2$ are conjugation invariant.
\item\label{lmCleanUnioncond2}
$I_i\cap\overline{I}_j=\emptyset$ for $1\leq i,j\leq2.$
\item\label{lmCleanUnioncond3}
$I_1=\overline{I}_2$.
\item\label{lmCleanUnioncond4}
$I_1\subset I_2$ or $I_2 \subset I_1.$
\end{enumerate}
\end{lm}
\begin{proof}
If condition~\ref{lmCleanUnioncond1} holds then $I_1\cup I_2$ is conjugation invariant and therefore clean. If condition~\ref{lmCleanUnioncond2} holds then
\[
I_1\cup I_2\cap\overline{I_1\cup I_2}=\bigcup_{1\leq i,j\leq 2}I_i\cap \overline{I}_j=\emptyset.
\]
So, $I_1\cup I_2$ is again clean. That conditions~\ref{lmCleanUnioncond3} and~\ref{lmCleanUnioncond4} imply cleanness of $I_1\cup I_2$ is obvious.

Conversely, suppose $I_1\cup I_2$ is clean. We divide into the case where $I_1$ is conjugation invariant and the case where it is not. If $I_1$ is conjugation invariant then in particular $I_1\cup I_2\cap\overline{I_1\cup I_2}\neq\emptyset$. So, by cleanness, $I_1\cup I_2=\overline{I_1\cup I_2}$. Now, if condition \ref{lmCleanUnioncond4} holds we are done. So we may assume that $I_2\setminus I_1\neq\emptyset$. Let $p\in I_2\setminus I_1$. Then, since $I_1=\overline{I_1}$, ${p}\in\overline{I_1\cup I_2}\setminus\overline{I}_1\subset\overline{I}_2$. In particular, $I_2\cap\overline{I}_2\neq\emptyset$. By cleanness of $I_2$ this implies $I_2$ is also conjugation invariant, so condition~\ref{lmCleanUnioncond1} holds.

Next we consider the case where $I_1$ is not conjugation invariant. If $I_2$ is conjugation invariant, exchanging the roles of $I_1$ and $I_2$ in the previous paragraph we deduce that condition~\ref{lmCleanUnioncond4} holds and we are done. Suppose now that $I_2$ is not conjugation invariant and consider $I_1\cup I_2$.  If $I_1\cup I_2$ is conjugation invariant, cleanness and non conjugation invariance of $I_1$ and $I_2$ imply that
\[
I_1\subset\overline{I}_2\setminus \overline{I_1}\subset \overline{I}_2
\]
and, similarly, $I_2\subset\overline{I}_1$. By conjugation invariance of the inclusion of sets, this implies Condition \ref{lmCleanUnioncond3}. If, on the other hand, $I_1\cup I_2$ is not conjugation invariant, cleanness implies that $I_1\cup I_2\cap\overline{I_1\cup I_2}=\emptyset$. So Condition \ref{lmCleanUnioncond2} holds.
\end{proof}

\begin{lm}\label{lmCleanMC}
Let $\Sigma$ be a connected Riemann surface with non-empty boundary. Let $g=genus(\Sigma_{\C})$ and let $\gamma\subset\Sigma_{\C}$ be a simple closed geodesic in $\Sigma_\C$.
\begin{enumerate}
\item
For $g\geq 2$, $\mathcal{C}(\gamma)$\footnote{See Definition \ref{dfCNeighbGeo}.} is clean if $\ell(\gamma)<2\sinh^{-1}(1)$.
\item\label{lmCleanMC2}
For $g\leq1$, $\mathcal{C}(\gamma)$ is clean if and only if either $\gamma\subset\partial\Sigma$, or
\[
\gamma\cap\partial\Sigma\neq\emptyset
\]
 and $\gamma\perp\partial\Sigma$.
\end{enumerate}

\end{lm}
\begin{proof}
\begin{enumerate}
\item
Since conjugation is an isometry we have that $\overline{\gamma}$ is also a simple closed geodesic, $\ell(\gamma)=\ell(\overline{\gamma})$ and $\overline{\mathcal{C}(\gamma)}=\mathcal{C}(\overline{\gamma})$. From Theorem \ref{TmThTh}\ref{it:coll} it therefore follows that $\mathcal{C}(\gamma)\cap\overline{\mathcal{C}(\gamma)}\neq\emptyset$ if and only if $\gamma\cap\overline{\gamma}\neq\emptyset$. Thus it suffices to prove that $\gamma$ is clean for $\gamma$ short enough. Suppose $\gamma\cap\overline{\gamma}\neq\emptyset$. If $\gamma\neq\overline{\gamma}$ then $\gamma$ intersects $\overline{\gamma}$ transversally. Therefore, by \cite[4.1.2]{Bu}, $\ell(\gamma)\geq 2\sinh^{-1}(1)$.
\item
By definition of $\mathcal{C}(\gamma)$ for this case, it is open and dense in $\Sigma$. Therefore, we  always have $\mathcal{C}(\gamma)\cap\overline{\mathcal{C}(\gamma)}\neq\emptyset$. Thus, $\mathcal{C}(\gamma)$ is clean if and only if $\mathcal{C}(\gamma)$ is conjugation invariant. That is, since conjugation is an isometry, if and only if $\gamma=\overline{\gamma}$. We claim that this is equivalent to the condition of the Lemma. For this it suffices to show that if $\gamma\not \subset\partial\Sigma,$ then $\gamma=\overline{\gamma}$ if and only if $\gamma\cap\partial\Sigma\neq\emptyset$ and $\gamma\perp\partial\Sigma$.

Indeed, if $\gamma\cap\partial\Sigma=\emptyset,$ then since $\gamma$ is connected it is contained in one component $\Sigma\setminus\partial\Sigma$ and is thus not conjugation invariant. So we assume $\gamma \cap \partial \Sigma \neq \emptyset.$ Let $p\in\gamma\cap\partial\Sigma$ and let $v$ be a vector tangent to $\gamma$ at $p$. Since $p$ is fixed under conjugation and since both $\gamma$ and $\overline{\gamma}$ are geodesics, we have that $\gamma=\overline{\gamma}$ if and only if $\overline{v}$ is also tangent to $\gamma$ at $p$. But $\gamma$ intersects $\partial\Sigma$ transversally since they are distinct geodesics. Therefore $v\neq\overline{v}$. Since $T_p\gamma$ is one dimensional it follows that $\overline{v}$ is tangent to $\gamma$ if and only if $\overline{v}=-v$. That is, $\overline v$ is tangent to $\gamma$ if and only if $v$ points in the direction of the imaginary axis in $T_p\Sigma_{\C}$. Since $T_p\partial\Sigma$ is the real axis, the claim follows.
\end{enumerate}
\end{proof}

\begin{lm}\label{lmCleanConjEmb}
Let $\Sigma$ be a Riemann surface. Let $I_1$ and $I_2$ be doubly connected and clean subsets of $\Sigma_{\C}$ which do not contain a component of $\partial\Sigma$. Suppose  $I_1\hookrightarrow I_2$ homologically nontrivially. Then $I_2$ is conjugation invariant if and only if $I_1$ is conjugation invariant.
\end{lm}
\begin{proof}
First we claim that $I_1$ is conjugation invariant if and only if each component of $I_1\setminus\partial\Sigma$ is simply connected. Assume $I_1$ is not conjugation invariant. Then since $I_1$ is clean, we have $I_1\setminus\partial\Sigma=I_1$. So, $I_1$ is the only connected component and is not simply connected. Conversely, assume $I_1$ is conjugation invariant. Suppose by contradiction that for one component $A$ of ${\Sigma}_{\C}\setminus\partial\Sigma$, there is a component of $I_1\cap A$ that is not simply connected. Since $I_1\cap A$ is isometric to $I_1\cap\overline{A}$, it is homotopy equivalent to it. Since $I_1$ does not contain any component of $\partial\Sigma$, each component of $I_1\cap\partial\Sigma$ is contractible.
Thus the Mayer Vietoris sequence implies that $I_1$ is at least two connected. This is a contradiction.

Now we prove the lemma. Assume $I_1$ is not conjugation invariant. Then since $I_1$ is clean, we have $I_1\cap \overline{I}_1=\emptyset$. Since $I_2\subset I_1$, this implies $I_2\cap\overline{I}_2=\emptyset$. Conversely, assume by contradiction that $I_1$ is conjugation invariant and $I_2$ is not. Let $A$ be the connected component of $\Sigma_{\C}$ containing $I_2$. $I_2$ is then contained in $I_1\cap A$ which by the previous paragraph is simply connected. This contradicts the fact that $I_2$ is embedded non-trivially in $I_1$.
\end{proof}
\section{Thick thin measure}
For the rest of the discussion, fix constants $c_1,c_2,c_3,\delta_1,\delta_2>0$ such that $c_3\leq 1$ and that $\delta_2<\frac12\delta_1$.

\begin{df}\label{dfThTh}
Let $(\Sigma,j)$ be a Riemann surface, possibly bordered. Let $\mu$ be a finite measure on $\Sigma$ and extend $\mu$ to a measure on $\Sigma_{\C}$ by reflection. That is,
\[
\mu(U):=\mu(\overline{U}),
\]
for $U\subset\overline{\Sigma}$  a measurable set. Suppose further that $\mu$ is absolutely continuous and has a continuous density $\frac{d\mu}{d\nu_h}$,  where $h$ is any Riemannian metric on $\Sigma_{\C}$.

The measure $\mu$  will be called \textbf{thick thin} if it satisfies the following two conditions.
\begin{enumerate}
\item \textbf{gradient inequality}. Let $U\subset\Sigma_{\C}$ be biholomorphic to the unit disk such that $U \cap \partial\Sigma$ is connected, and let $z\in U$. Then for any conformal metric $h$ on $(\Sigma_{\C},j)$,
\begin{align}\label{GradEq}
 \mu(U)<\delta_1\quad \Rightarrow \quad &\frac{d\mu}{d\nu_h}(z)\leq c_1\frac{\mu(U)}{r_{conf}^2},
\end{align}
where $r_{conf}=r_{conf}(U,z;h)$.
\item \textbf{cylinder inequality}. Let $I\subset\Sigma_{\C}$ be clean and doubly connected such that $Mod(I)> 2c_2$. Then for all $t\in \left(c_2,\frac1{2}Mod(I)\right)$ we have,
\begin{gather}
\notag \mu(I)<\delta_2 \qquad \Rightarrow \qquad\mu(C(t,t;I))\leq e^{-c_3t}\mu(I)\notag.
\end{gather}
\end{enumerate}
\end{df}

A family of measured Riemann surfaces which are thick thin with respect to given constants $c_i$, $\delta_i$ will be referred to as a uniformly thick thin family.

\begin{rem}\label{rmGrad}
Let $\mu$ be a thick thin measure on $\Sigma,$ and let $h$ be a conformal metric of constant curvature $K = 0,\pm 1$ on $\Sigma_\C$. By inequalities~\eqref{eq:bdrc1} and~\eqref{eq:bdrc2}, there is a constant $c'_1$ depending linearly on $c_1$ such that for any $z\in\Sigma$ and $r\in(0,\min(\sinh^{-1}(1),\inj(\Sigma;h,z)))$,
\begin{align}
\mu(B_r(z;h))<\delta_1\Rightarrow &\frac{d\mu}{d\nu_h}(z)\leq c'_1\frac{\mu(B_r(z;h))}{r^2}.
\end{align}
\end{rem}

Let $\Sigma,\mu$ and $h$ be as in Remark~\ref{rmGrad}. For any point $z \in \Sigma_\C$, let $d= \frac{d\mu}{d\nu_h}(z)$ and let
\begin{align}\label{EqArDee}
r_d:=\sqrt{\frac{c'_1\delta_1}{d}}.
\end{align}
\begin{lm}\label{rmGradApp}
Suppose
\begin{equation*}
r_d \in (0,\min(\sinh^{-1}(1),\inj(\Sigma;h,z))).
\end{equation*}

Then
\[
\mu(B_{r_d}(z;h))\geq\delta_1.
\]

Moreover, we have
\[
\frac{d\mu}{d\nu_h}(z) \leq c_1' \frac{\mu(B_r(z;h))}{r^2}, \qquad r \leq r_d.
\]
\end{lm}
\begin{proof}
This is immediate from the gradient inequality.
\end{proof}
To simplify our formulas, we  scale $\mu$ so that $c'_1\delta_1=1$.

We denote by $\mathcal{M}=\mathcal{M}(c_1,c_2,c_3,\delta_1,\delta_2)$ the family of measured Riemann surfaces $(\Sigma,j,\mu)$ such that $\mu$ is thick-thin.

\begin{lm}\label{ExpCylDEst}
There is a constant $a$ with the following significance. Let $(\Sigma,\mu)\in\mathcal{M}$. Let $I\subset\Sigma_{\C}$ be clean and doubly connected, and let $h = h_{st}.$ Suppose $\mu(I)<\delta_2$. Let $z\in C(c_2+\pi,c_2+\pi;I)$ be a point with cylindrical coordinates
\[
(\rho,\theta)\in\left[-\frac1{2}Mod(I)+c_2+\pi,\frac1{2}Mod(I)-c_2-\pi\right]\times S^1.
\]
Then,
\begin{align}
 \frac{d\mu}{d\nu_{h}}(z)<ae^{-c_3(\frac1{2}Mod(I)-|\rho|)}\mu(I).
\end{align}
\end{lm}
\begin{proof}
Combining the gradient inequality and the cylinder inequality,
\begin{align}
 \frac{d\mu}{d\nu_{h}}(z)&\leq \frac{c_1}{\pi^2}\mu\left([\rho-\pi,\rho+\pi]\times S^1\right)\\&\leq \frac{c_1}{\pi^2}\mu\left([-|\rho|-\pi,|\rho|+\pi]\times S^1\right)\notag\\&\leq \frac{c_1}{\pi^2}e^{-c_3(\frac1{2}Mod(I)-\pi-|\rho|)}\mu(I)\notag.
\end{align}
\end{proof}

\begin{df}\label{dfmuStable}
Let $(\Sigma,\mu)\in \mathcal{M}$. A connected compact sub-manifold with boundary $A\subset\Sigma_\C$ is said to be \textbf{$\mu$-stable} if one of the following holds:
\begin{enumerate}
\item
$\mu(A)\geq\delta_1/2$;
\item
$\#\pi_0(\partial A)\geq 2$ and $\mu(A)\geq\delta_2/6;$
\item
$
2\,genus(A) + \#\pi_0(\partial A)\geq 3.
$
\end{enumerate}
A compact sub-manifold with boundary $A\subset\Sigma_C$ is said to be $\mu$-stable if each of its connected components is $\mu$-stable.
\end{df}
\section{Preparation}

\subsection{Choice of spherical metric}
As explained in Remark~\ref{remSphChoice}, the genus $0$ case requires a non-trivial choice of Fubini-Study metric. This is done in the following lemma.

\begin{lm}\label{LmSpSphMet}
There is a constant $K_0$ with the following property. For each $(\Sigma,\mu) \in \mathcal M$ with $genus(\Sigma_{\C})=0$ and $\mu(\Sigma_\C)\neq 0,$ there exists a conjugation invariant unit curvature Fubini-Study metric, $h,$ on $\Sigma_{\C}$ such that one of the following conditions holds.
\begin{enumerate}
\item\label{LmSpSphMet0}
 \[
 \sup_{\Sigma_\C}\frac{d\mu}{d\nu_h}\leq K_0.
\]
\item\label{LmSpSphMet1}
There is a point $q\in\Sigma$ such that
\[
\mu(B_{\pi/2}(q;h))\geq\delta_1
\]
 and
 \[
 \frac{d\mu}{d\nu_h}\Big|_{B_{\pi/2}(q;h)}\leq K_0.
\]
 Furthermore, if $\partial\Sigma\neq\emptyset,$ then $q\in\partial\Sigma$.
\item\label{LmSpSphMet2}
$\partial\Sigma\neq\emptyset$. Use $h$ to identify $\Sigma_\C$ with the standard sphere in such a way that $\partial\Sigma$ is identified with the equator and let $q$ be the north pole. Letting $d:=\frac{d\mu}{d\nu_h}(q),$ we have $r_{d} \leq\pi/4$, and
\[
\sup_{x\in B_{r_d}(q)}\frac{d\mu}{d\nu_h}(x)\leq K_0d.
\]
\end{enumerate}
Furthermore, for any disc $D\subset\Sigma_\C$ of radius $\min\left\{\sqrt{\frac{\delta_1/2}{\pi K_0}},\pi/4\right\}$, the complement $\Sigma_\C\backslash D$ is stable.
\end{lm}
\begin{rem}\label{RemPreCCl}
When $\partial\Sigma=\emptyset$ we have that $\Sigma_\C$ has two components: $\Sigma$ and $\overline{\Sigma}$. The conjugate component is if no interest. In the sequel we shall avoid talking about $\Sigma_\C$ in the closed context.
\end{rem}
\begin{rem}
The three cases are correspond to the quantitative counterparts of the possible behaviors of the Gromov limit of a sequence of genus 0 curves:
\begin{enumerate}
\item
 No bubbling.
 \item
Bubbling off of spheres or disks without the boundary degenerating.
 \item
 Bubbling in which the boundary degenerates to a point.
 \end{enumerate}
\end{rem}
\begin{proof}
Choose initially any conjugation invariant unit curvature metric $h_0$ on $\Sigma_\C$. Let $p\in\Sigma$ be a point where the maximum of $\frac{d\mu}{d\nu_{h_0}}$ is obtained, and let
\[
d_0:=\frac{d\mu}{d\nu_{h_0}}(p).
\]
If $r_{d_0}>\pi/6$, then the estimate
\[
\frac{d\mu}{d\nu_{h_0}}\leq \frac{36}{\pi^2},
\]
holds globally. Thus, condition~\ref{LmSpSphMet0} holds with $K_0=\frac{36}{\pi^2}$.

Assume
\begin{equation}\label{LmSpSphAssu}
r_{d_0}\leq\pi/6.
\end{equation}
If $\partial\Sigma=\emptyset,$ let $p_1:=p$. Otherwise, let $p_1$ be the midpoint of a length minimizing geodesic which connects $p$ and $\overline{p}.$ Let $p_2$ be the antipode of $p_1$ with respect to $h_0$. With respect to $h_0$, let
\[
(\rho,\theta):\Sigma_\C\setminus\{p_2\}\to\C
\]
be geodesic polar coordinates centered at $p_1$. In case $\partial\Sigma=\emptyset,$ assume further that $\partial\Sigma\setminus\{p_2\}$ is given by $\{\theta=0\}.$ Let $\phi:\Sigma_\C\setminus \{p_2\}\to \C$  be stereographic projection. Explicitly, in polar coordinates $\phi$ is given by
\begin{equation}
(\rho,\theta)\mapsto \tan\frac{\rho}{2}e^{i\theta}.
\end{equation}
Note that $p$ and $\overline{p}$ are mapped by $\phi$ to the imaginary axis.

Let $r=d(p,p_1;h_0)$. Suppose first that
\begin{equation}\label{r0leqrd}
r<2r_{d_0}.
\end{equation}
We prove that condition \ref{LmSpSphMet1} holds. Let $\chi:\C\to\C$ be the map
\[
z\mapsto \frac{z}{\tan\left(\frac{r+r_{d_0}}{2}\right)}.
\]
Let  $\psi:\Sigma_\C\to\Sigma_\C$ be the holomorphic map defined by
\[
\psi\big|_{\Sigma_\C\setminus{p_2}}=\phi^{-1}\circ\chi\circ\phi,
\]
and let $h_1:=\psi^*{h_0}$. Note that the change of metric from $h_0$ to $h_1$  scales the disc of radius $r+r_{d_0}$ around $p_1$ to become the hemisphere centered at $p_1.$ In particular, $B_{r_{d_0}}(p;h_0)\subset B_{\pi/2}(p_1;h_1)$. So, by Lemma~\ref{rmGradApp},
\[
\mu(B_{\pi/2}(p_1;h_1))\geq\delta_1.
\]
We show now that the energy density is bounded on the hemisphere centered at $p_1$, uniformly in $\mathcal{M}$. First note that for any $z\in\Sigma_\C$,
\begin{align}\label{eqderpsest}
\frac{d\mu}{d\nu_{h_1}}(z)=\frac{d\mu}{d{\nu_{{h_0}}}}(z)\frac{d\nu_{h_0}}{d\nu_{h_1}}(z)={\|d\psi\|^{-2}_{h_0}(z)}\frac{d\mu}{d\nu_{h_0}}(z).
\end{align}
A computation gives
\begin{align}\label{eqdpsicalc}
\|d\psi\|_{h_0}(x)=\frac{\tan{\frac{r+r_{d_0}}{2}}}{\cos^2(\rho(x)/2)\tan^2{\frac{r+r_{d_0}}{2}}+ {\sin^2(\rho(x)/2)}}.
\end{align}

Assumptions~\eqref{LmSpSphAssu} and~\eqref{r0leqrd} imply that $\|d\psi\|_{h_0}^{-1}$ increases with distance from $p_1$ on the ball ${B_{\pi/2}(p_1;h_1)}$. In particular,
\begin{equation}\label{eq:sin}
\sup_{B_{\pi/2}(p_1;h_1)} \|d\psi\|_{h_0}^{-1} = \sin(r + r_{d_0}).
\end{equation}
Using equations \eqref{r0leqrd}, \eqref{eqderpsest}, and the definition of $r_{d_0}$, we get
\[
\sup_{B_{\pi/2}(p_1;h_1)} \frac{d\mu}{d\nu_{h_1}}\leq K_0,
\]
for an appropriate constant $K_0$ which is independent of $\mu$. This is condition~\ref{LmSpSphMet1} with $h=h_1$ and $q=p_1.$

Now suppose
\begin{equation}\label{eqrogeqrd}
r\geq 2r_{d_0}.
\end{equation}
Let $\chi:\C\to\C$ be the map
\[
z\mapsto \frac{z}{\tan(\frac{r}{2})}.
\]
Let  $\psi:\Sigma_\C\to\Sigma_\C$ be the holomorphic map defined by
\[
\psi\big|_{\Sigma_\C\setminus\{p_2\}}=\phi^{-1}\circ\chi\circ\phi,
\]
and let $h_1:=\psi^*{h_0}$. Write $d_2:=\frac{d\mu}{d\nu_{h_1}}(p)$, $A:=B_{r_{d_2}}(p;h_1)$, and
\[
C:=\frac{\|d\psi\|_{h_0}^2(p)}{\inf_{w\in A} \|d\psi\|_{h_0}^2(w)}.
\]
Then we have the bound
\[
\frac{d\mu}{d\nu_{h_1}}\Big|_A\leq Cd_2.
\]

Note that $\|d\psi\|_{h_0}$ is obtained by substituting $r$ in place of $r+r_{d_0}$ in equation~\eqref{eq:sin}, and that $r\leq\pi/2$. Therefore, $\|d\psi\|_{h_0}$ is decreasing for $\rho(x)\in[0,\pi].$ Let $x_0$ be the point which maximizes $\rho(x)$ on $A$. One computes that
\[
C=\frac{\cos^2(\rho(x_0)/2)}{2\cos^2{r/2}}+\frac{{\sin^2(\rho(x_0)/2)}}{2{\sin^2(r/2)}}.
\]
To bound $C$ it suffices to bound the ratio $\frac{\rho(x_0)}{r}.$ By direct computation,
\[
\rho(x_0)= 2\tan^{-1}\left(\tan {\frac{r}{2}}\tan{\frac{\pi/2+r_{d_2}}{2}}\right).
\]
Note now that $r_{d_2}=\frac{r_{d_0}}{\sin r}$. Using assumption~\eqref{eqrogeqrd} and the fact the function $r \mapsto \frac{r}{2\sin r}$ is monotone increasing for $0<r<\pi$ and that $r\leq\pi/2$,  we conclude that $r_{d_2}\leq\pi/4$. Thus, $\frac{\rho(x_0)}{r}\leq C'$ for some uniform constant $C'$. There is therefore an a priori constant $K_0$ bounding $C$. This gives condition~\ref{LmSpSphMet2} with $h=h_1$ and $q=p.$

We prove the last part of the claim. Suppose condition~\ref{LmSpSphMet0} holds. As is well known, the gradient inequality implies
\[
\mu(\Sigma_\C)\neq 0\Rightarrow \mu(\Sigma_\C)\geq\delta_1.
\]
It is straightforward to verify that $\mu(D)\leq\delta_1/2$, implying the claim. If condition~\ref{LmSpSphMet1} holds, one similarly verifies that
\[
\mu(D\cap B_{\pi/2}(q;h))\leq\delta_1/2.
\]
So, for each component\footnote{See remark~\ref{RemPreCCl}.} $\Sigma'$ of $\Sigma_\C,$
\[
\mu(\Sigma'\setminus D)\geq\mu(B_{\pi/2}(q;h)\setminus D)\geq\delta_1/2.
\]
Suppose now that condition~\ref{LmSpSphMet2} holds. Then $D$ meets at most one of the discs $B_1=B_{r_d}(q;h)$ and $B_2=B_{r_d}(\overline{q};h)$. By Lemma~\ref{rmGradApp}, $\mu(B_i)\geq\delta_1$ for $i=1,2$.
\end{proof}

\subsection{Admissible annuli}
From now to Section~\ref{SEcMuAdapt}, we fix a $(\Sigma,\mu)\in\mathcal{M}$. However, all constants are that appear in the sequel are independent of $\Sigma$ and $\mu$. Let
\[
\Sigma'=\begin{cases}
\Sigma, & \partial\Sigma=\emptyset,\\
\Sigma_\C, & \partial\Sigma\neq\emptyset.
\end{cases}
\]
If $genus(\Sigma')\geq 2$, let $h$ be the unique conformal metric of constant curvature $-1$ on $\Sigma'$. If $genus(\Sigma')=1$, let $h$ be the unique conformal metric of constant curvature $0$ and of unit area. Finally, if $genus(\Sigma')=0$, let $h$ be a conformal metric of constant curvature 1 which satisfies the property of Lemma \ref{LmSpSphMet}.

Our goal in the following three sections is to construct a $(\mu,h)$-adapted bubble decomposition of $\Sigma_\C$. We make the following assumption
\begin{as}\label{as1}
$(\Sigma,\mu)$ satisfies one of the following:
\begin{enumerate}
\item
$genus(\Sigma')=0$ and $h$ does not satisfy condition \ref{LmSpSphMet0} in Lemma \ref{LmSpSphMet}.
\item
\label{eq:deg}
\begin{equation}
genus(\Sigma')=1,~ and~\mu(\Sigma')>\delta_2\notag.
\end{equation}
 \item
 \begin{equation}
 genus(\Sigma')>1\notag.
 \end{equation}
\end{enumerate}
\end{as}
The cases not covered Assumption~\ref{as1} are referred to as the \textbf{trivial cases}. The genus 0 trivial case automatically admits a $(\mu,h)$-adapted bubble decomposition and  requires no treatment. The trivial genus 1 case will be treated separately in the proof of Theorem~\ref{tmBubDecEst1}.

We need to partially break the symmetry in the cases of genus 0 and 1 in Definition \ref{dfAdmAnn} below. For this we introduce some notation. Suppose $genus(\Sigma')=0$. If in Lemma \ref{LmSpSphMet} condition \ref{LmSpSphMet1} holds for $h$, let $q$ be the point given there and let $\tilde{\Sigma}:=\Sigma'\setminus\{q\}$. If, instead, condition~\ref{LmSpSphMet2} holds, let $q$ be the point given there and let  $\tilde{\Sigma}:=\Sigma'\setminus\{q,\overline{q}\}$. For $genus(\Sigma')\geq 1$, let $\tilde{\Sigma}:=\Sigma'$.

Suppose now $genus(\tilde{\Sigma})=1$. Our normalization of $h$ implies there is at most one element of $H_1(\tilde{\Sigma};\mathbb{Z})$ which is represented by a simple geodesic of length less than 1. Suppose such a class exists, and denote it by $A$. Pick closed geodesics $\alpha_0$ and $\alpha_1$ representing $A$ as follows. If $\partial\Sigma\neq\emptyset$ and the components of $\partial\Sigma$ represent $A$, let $\alpha_0$ and $\alpha_1$ be the components of $\partial\Sigma$. Otherwise,  let $I_0$ be a sub-cylinder maximizing the modulus among all the subcylinders $I$ such that $\mu(I)=\delta_2/6$ and each component of $\partial I$ represents $A$. Fix a biholomorphism
\[
f:I\to \left[-\frac12Mod(I_0),\frac12Mod(I_0)\right],
\]
and let $\alpha_0:=f^{-1}(\{0\}\times S^1)$. Let $\alpha_1$ be the geodesic whose image is $\tilde{\Sigma}\setminus\mathcal{C}(\alpha_0)$. $I_0$ will play a role in the construction of the bubble decomposition. We therefore define it also when $\partial\Sigma\neq\emptyset.$ In this case, define $I_0\subset \mathcal{C}(\alpha_0)$ to be a conjugation invariant sub-cylinder containing $\alpha_0$ and satisfying $\mu(I_0)=\delta_2/6$.

\begin{df}\label{dfAdmAnn}
An admissible annulus is a doubly connected clean open $I\subset\tilde{\Sigma}$ of one of the following forms:
\begin{enumerate}
\item\label{dfAdmAnn1}  There is a point $z\in\tilde{\Sigma}$ and positive reals $r\in(0,\frac1{3}\inj(\tilde{\Sigma};h,z)]$ and $r'\in(0,\frac1{5}r]$ such that $I= A(r,r',z;h)$. Furthermore, $I$ is contractible in $\tilde{\Sigma}$\footnote{This is of course redundant when $genus(\tilde{\Sigma})>0.$}.
\item
In case $genus(\tilde{\Sigma})=1$ assume $\alpha_1$ is defined. In case $genus(\tilde{\Sigma})=0$ assume $\partial\Sigma\neq \emptyset.$ There is a simple closed geodesic $\gamma\subset \tilde{\Sigma}$ satisfying
\[
\begin{cases}
\ell(\gamma)<2\sinh^{-1}(1), & genus(\tilde{\Sigma})> 1,\\
\gamma=\alpha_1, & genus(\tilde{\Sigma})= 1,\\
\gamma=\partial\Sigma, & genus(\tilde{\Sigma})=0
\end{cases}
\]
such that $I$ is an open sub-cylinder\footnote{See Definition~\ref{dfSubCyinlder}} of $\mathcal{C}(\gamma)$\footnote{See Definition~\ref{dfCNeighbGeo}}.
\end{enumerate}
If $I$ is of the type~\ref{dfAdmAnn1} it will be referred to as a trivial admissible annulus. Otherwise, it will be referred to a non trivial admissible annulus. We will also use the term admissible cylinder for nontrivial annuli. We denote by $\mathcal{A}_h$ the collection of admissible annuli both trivial and non trivial.

When $genus(\tilde{\Sigma})=1$ and $\alpha_0$ is defined, we will also use the notation $\hat{\mathcal{A}}_h$ for the union of $\mathcal{A}_h$ with the set of sub-cylinders of $\mathcal{C}(\alpha_0)$. In all other cases, $\hat{\mathcal{A}}_h:=\mathcal{A}_h.$
\end{df}

\begin{rem}\label{remAdmMot}
Note that an admissible trivial annulus is uniquely representable as the difference between two discs in $\tilde{\Sigma}$. Henceforth, whenever we represent an admissible annulus $I$ as the difference $I=B\setminus B'$, it is intended that $B'\subset \tilde{\Sigma}$.
\end{rem}
\begin{rem}\label{COrientConv}
Recall our notation $C(a,b;I)$ for an annulus $I$ and reals $a,b$. When $a\neq b$ this notation is well defined only up to a holomorphic reflection since it depends on the choice of holomorphic parametrization
\[
(\rho,\theta):\left[0,Mod(I)\right]\times S^1\to I.
\]
We adopt the convention that for a trivial annulus $I=B\setminus B'$, $\rho$ increases as the distance to the center of $B_1$ increases. For nontrivial annuli we assume that for each simple closed geodesic we fixed a choice of holomorphic parametrization of $\mathcal{C}(\gamma)$ by
\[
 \left[0,Mod(\mathcal{C}(\gamma))\right]\times S^1
\]
once and for all. This induces a choice for all the admissible nontrivial annuli.
\end{rem}

\subsection{Topological relatedness}
\begin{df}\label{dfTopRel}
Let $I_1,I_2\in\mathcal{A}_h$. We say that $I_1$ and $I_2$ are \textbf{topologically related} if there exists a doubly connected clean $I\subset\tilde{\Sigma}$ such that both $I_1$ and $I_2$ are nontrivially embedded in $I$.
\end{df}

\begin{tm}\label{lmRelChar}
Let $I_1,I_2\in\mathcal{A}_h$. $I_1$ and $I_2$ are topologically related if and only if one of the following holds:
\begin{enumerate}
\item\label{lmRelChar1} There is a simple closed geodesic $\gamma$ such that $\mathcal{C}(\gamma)$ is clean and both $I_1$ and $I_2$ are sub-cylinders of $\mathcal{C}(\gamma)$. Furthermore, when $genus(\tilde{\Sigma})=1$, we have that $\alpha_1$ is defined and that $\gamma=\alpha_1$. When $genus(\tilde{\Sigma})=0$, we have that $\partial\Sigma\neq\emptyset$ and $\gamma=\partial\Sigma$.
\item\label{lmRelChar2} There are concentric geodesic discs $B'_i\subset B_i\subset\tilde{\Sigma}$ such that $I_i=B_i\backslash B'_i,$ for $i=1,2$. Furthermore,
    \begin{enumerate}
    \item
    $B'_1\cap B'_2\neq\emptyset$,
    \item
     $I_1\cup I_2$ is clean.
    \end{enumerate}
 \end{enumerate}
 \end{tm}
To prove Theorem~\ref{lmRelChar} we first prove the following Lemmas some of which will also be used later.

\begin{lm}\label{lmCleanannB}
Let $I=B\setminus B'$ be an admissible trivial annulus. Then both $B$ and $B'$ are clean. Furthermore, $I$ is conjugation invariant if and only if both $B$ and $B'$ are.
\end{lm}
\begin{proof}
If $I\cap\overline{I}=\emptyset$ there is a component $A$ of $\tilde{\Sigma}\setminus\partial\Sigma$ such that $I\subset A$. Suppose by contradiction that $B\not\subset A$ then $B\cap\partial\Sigma\neq\emptyset$. Since $\partial B\subset\partial I$ this implies there is a component of $\partial\Sigma$ contained in $B$. When $genus(\tilde{\Sigma})>0$ this is a contradiction since $B$ is contractible whereas each component of $\partial\Sigma$ is a closed geodesic. When $genus(\tilde{\Sigma})=0$ it is straightforward to verify by definition that $I$ cannot be admissible. Thus $B\cap\overline{B}=\emptyset$. Since $B'\subset B$ the same is true for $B'$.

If $I=\overline{I}$ then $\partial I$ is conjugation invariant. Thus, either each component of $\partial I$ is conjugation invariant, or each component of $\partial I$ is contained in different component of $\tilde{\Sigma}\setminus\partial\Sigma$. But this latter case is ruled as in the previous paragraph. In particular we get that $\partial B$ and $\partial B'$ are each conjugation invariant. So, the same is true for $B$ and $B'$. Thus we have proven the first part of the lemma and one direction of the second part. The other direction is obvious.
\end{proof}
\begin{lm}\label{lmBallsContr}
Let $p_i\in\tilde{\Sigma}$ and $r_i\in(0,\frac1{3}\inj(\tilde{\Sigma};h,p_i))$ and write $B_i=B_{r_i}(p_i;h)$ for $i=1,2$. If $B_1\cap B_2\neq\emptyset$, then
\begin{enumerate}
\item \label{it:lmBallsContr1} $B_1\cup B_2$ is contained in a geodesic disc.
\item \label{it:lmBallsContr2} The closure of $B_1 \cap B_2$ is homeomorphic to the closed disc.
\item \label{it:lmBallsContr3} The closure of $B_1\cup B_2$ is homeomorphic to the closed disc.
\end{enumerate}
\end{lm}
\begin{proof}
Let
\[
r=\inj(\Sigma;h,p_1),
\]
and without loss of generality assume

 \[
r\geq \inj(\Sigma;h,p_2).
\]
 Then $r_2< \frac13 r$ and so, since $d(p_1,p_2)\leq r_1+r_2<\frac23 r$, we have
 \[
 B_1\cup B_2\subset B_r(p_1).
 \]
 This gives part~\ref{it:lmBallsContr1}. Since the curvature is constant, the sizes of the $r_i$ imply that the balls $B_i$ are geodesically convex. Thus, $B_1\cap B_2$ is geodesically convex and therefore simply connected. It follows from Van Kampen's theorem that $B_1\cup B_2$ is also simply connected. Clearly, the closures of $B_1\cup B_2$ and $B_1\cap B_2$ are topological surfaces with boundary. Parts~\ref{it:lmBallsContr2} and~\ref{it:lmBallsContr3} follow.
\end{proof}

\begin{lm}\label{lmRelDCon}
Let $I_i=A(r_i,r'_i;p_i)\in\mathcal{A}_h$ for $i=1,2$. Suppose
\[
B^c_{r'_1}(p_1)\cap B^c_{r'_2}(p_2)\neq\emptyset.
\]
Let
\[
I=B_{r_1}(p_1)\cup B_{r_2}(p_2)\setminus B^c_{r'_1}(p_1)\cap B^c_{r'_2}(p_2).
\]
Then $I$ is doubly connected. Furthermore, if $I_1\cup I_2$ is clean, so is $I$.
\end{lm}
\begin{proof}
Write $B_i=B_{r_i}(p_i)$ and $B'_i=B^c_{r'_i}(p_i)$ for $i=1,2$. By Lemma~\ref{lmBallsContr} parts~~\ref{it:lmBallsContr2} and~\ref{it:lmBallsContr3}, $B'_1\cap B'_2$ is homeomorphic to the closed disc and $B_1\cup B_2$ is homeomorphic to the open disc. Denoting by $I^c$ the closure of $I$, the Mayer Vietoris sequence implies that $H_1(I^c;\mathbb{Z})=\mathbb{Z}$. The only orientable surface with boundary satisfying this is the annulus.

To see that $I$ is clean if $I_1\cup I_2$ is, distinguish between the possibilities for $I_1$ and $I_2$ according to Lemma \ref{lmCleanUnion}.
\begin{enumerate}
\item
$I_i$ is conjugation invariant for $i=1,2$. By Lemma \ref{lmCleanannB}, so are $B_i$ and $B'_i$ and therefore, so is $I$.
\item
$I_i\cap\overline{I_j}=\emptyset$ for $i,j=1,2$. By Lemma \ref{lmCleanannB}, $B_i\cap\overline{B_i}=\emptyset$. Since $B_1\cap B_2\neq\emptyset$, $B_1$ and $B_2$ belong to the same component of $\tilde{\Sigma}\setminus\partial\Sigma$. Therefore,  $I\subset B_1\cup B_2$ belongs to one component of $\tilde{\Sigma}\setminus\partial\Sigma$. In particular, $I\cap\overline{I}=\emptyset$.
\item
$I_1=\overline I_2$. We claim that $I_1$ is conjugation invariant and so $I=I_1=I_2$. Indeed if $I_1$ is not conjugation invariant, then Lemma \ref{lmCleanannB} implies $B'_1$ is clean and not conjugation invariant.  So, $B'_1\cap \overline{B}_1' =\emptyset$. But since $I_1=\overline{I}_2$, $\overline{B}'_1=B'_2$. We thus get a contradiction to the assumption $B'_1\cap B'_2\neq\emptyset$.
\item
Without loss of generality $I_1\subset I_2$. In this case it is clear that $B'_2\subset B'_1$, so $I=I_2$.
\end{enumerate}
\end{proof}

 \begin{proof}[Proof of Theorem~\ref{lmRelChar}]
Assume Condition \ref{lmRelChar1} holds. Then $\mathcal{C}(\gamma)$ plays the role of $I$ in Definition \ref{dfTopRel}. Assume Condition \ref{lmRelChar2}, holds and let
\[
I=B_1\cup B_2\setminus (B'_1\cap B'_2).
\]
By Lemma \ref{lmRelDCon}, $I$ is clean and doubly connected.  Clearly, $I_i$ is nontrivially embedded in $I$ for $i=1,2$. Thus $I_1$ and $I_2$ are topologically related.

Conversely, let  $I_1$ and $I_2$ be embedded nontrivially in a clean $I\subset\tilde{\Sigma}$. $I_1$ and $I_2$ are homologous in $I$ to a homology generator of $I$. This implies that $I_1$ and $I_2$ are either both trivially embedded or both nontrivially embedded in $\tilde{\Sigma}$. These correspond to the cases where $I$ is embedded trivially and nontrivially respectively.

In the first case, if $genus(\tilde{\Sigma})\leq 1$ the only non-trivial annuli are the sub-cylinders of $\mathcal{C}(\alpha_1)$ and $\mathcal{C}(\partial\Sigma)$ which are clean. If $genus(\tilde{\Sigma})>1$, there are simple closed geodesics $\gamma_i$, for $i=1,2$, such that $I_i$ is a sub-cylinder of $\mathcal{C}(\gamma_i)$. We have that $\gamma_i$ is freely homotopic to any component of $\partial I_i$ which in turn is freely homotopic to any component of $\partial I$. So $\gamma_1$ is freely homotopic to $\gamma_2$. Since there is a unique simple closed geodesic in each free homotopy class, this implies $\gamma_1=\gamma_2$. Clearly, $\ell(\gamma)<2\sinh^{-1}(1)$ since $\mathcal{C}(\gamma)$ contains admissible cylinders as sub cylinders. Therefore, by Lemma~\ref{lmCleanMC}, $\mathcal{C}(\gamma)$ is clean.

Now consider the case where $I_i$ are trivial for $i=1,2$. Then $I$ must be trivially embedded in $\tilde{\Sigma}$. Since $\tilde{\Sigma}$ is not a sphere, $\tilde{\Sigma}\setminus I$ has exactly one component $A$ with the topology of a disc. Clearly, $A\subset B'_1\cap B'_2$. This gives the first part of~\ref{lmRelChar2}. Now, if $I\cap \overline{I}=\emptyset$ then clearly $I_1\cup I_2\subset I$ is clean. If $I$ is conjugation invariant then by Lemma \ref{lmCleanConjEmb}, so are $I_1$ and $I_2$. This implies that so is $I_1\cup I_2$, giving the second part of~\ref{lmRelChar2}.
\end{proof}

Let  $I_1$ and $I_2$ be topologically related. We associate with $I_1$ and $I_2$ two clean doubly connected sub-surfaces $M(I_1,I_2)$ and $m(I_1,I_2)$ in which both are non-trivially embedded. One should think of $M(I_1,I_2)$ as the minimal annulus in $\tilde{\Sigma}$ in which $I_1$ and $I_2$ are nontrivially embedded. On the other hand, $m(I_1,I_2)$ should be thought of as the maximal \textit{admissible} annulus which is nontrivially embedded in $M(I_1,I_2).$

Formally, the definitions are as follows. When $I_1$ an $I_2$ are sub-cylinders of $\mathcal{C}(\gamma)$ for a simple closed geodesic $\gamma$, suppose that $I_i$ is given in $(\rho,\theta)$ coordinates\footnote{See the discussion subsequent to Definition~\ref{dfCNeighbGeo}.} by
\[
I_i:=\left\{z\in\mathcal{C}(\gamma)|\rho_{0,i}<\rho(z)<\rho_{1,i}\right\}.
\]
Let $\rho_0=\min\left\{\rho_{0,1},\rho_{0,2}\right\}$ , $\rho_1=\max\left\{\rho_{1,1},\rho_{1,2}\right\}$, $\rho=\max\left\{|\rho_0|,|\rho_1|\right\}$. Define
\[
M(I_1,I_2):=\begin{cases}
\left\{z\in\mathcal{C}(\gamma)|-\rho<\rho(z)<\rho\right\},& \mbox{$\gamma\subset\partial\Sigma$ and $0\in[\rho_0,\rho_1]$},\\
\left\{z\in\mathcal{C}(\gamma)|\rho_0<\rho(z)<\rho_1\right\},&\mbox{otherwise.}
\end{cases}
\]
When $I_1$ and $I_2$ are trivial, write $I_i=B_i\setminus B'_i$ and take
\[
M(I_1,I_2):=B_1\cup B_2\setminus B'_1\cap B'_2.
\]

We now define $m(I_1,I_2)$. If $I_1$ and $I_2$ are nontrivial take $m(I_1,I_2):=M(I_1,I_2)$. Otherwise, suppose $I_i=B_{r_i}(p_i)\setminus B^c_{r'_i}(p_i)$ and assume without loss of generality that $r_2\leq r_1$. Then define
\[
m(I_1,I_2):=A\left(d(p_2,\partial B_1;h),r'_2,p_2;h\right).
\]
\begin{lm}\label{lmI1sharpI2}
$M(I_1,I_2)$ and $m(I_1,I_2)$ are clean and doubly connected.

\end{lm}
\begin{proof}
That $M(I_1,I_2)$ is clean and doubly connected follows from the definition, from Theorem \ref{lmRelChar}, and from Lemma~\ref{lmRelDCon}. To prove the same for $m(I_1,I_2)$ we may assume $I_1$ and $I_2$ are trivial as otherwise $m(I_1,I_2)=M(I_1,I_2)$. Clearly, $m(I_1,I_2)$ is doubly connected. We show that it is clean. If $m(I_1,I_2)\cap\overline{m(I_1,I_2)}=\emptyset$, we are done. Otherwise, let $I_i=B_i\setminus B'_i$ for $i=1,2$. We claim that $B_1$ and $B_2$ are each conjugation invariant. Indeed, since $m(I_1,I_2)\subset B_1\cup B_2$, we have that $B_1\cup B_2$ meets $\partial\Sigma$. Without loss of generality, $B_1$ meets $\partial\Sigma$. Since by Lemma \ref{lmCleanannB} $B_1$ is clean, it must be conjugation invariant. Thus, $I_1$ meets $\partial\Sigma$. So, $I_1$ is conjugation invariant. By definition of topological relatedness and by Lemma \ref{lmCleanConjEmb}, $I_2$ is also conjugation invariant. By Lemma \ref{lmCleanannB} again, $B_2$ is conjugation invariant. In particular, the centers $p_1,p_2$ of $B_1$ and $B_2$ lie on $\partial\Sigma$. It is now clear by construction that $m(I_1,I_2)$ is conjugation invariant.
\end{proof}

\begin{lm}\label{lmNonTrivEithOr}
Suppose the pairs $(I_1,I_2)$ and $(I_2,I_3)$ are topologically related.
\begin{enumerate}
\item
If $I_i$ is trivial for $i=1,2,3,$ then
\[
M(I_1,I_3)\subset M(I_1,I_2)\cup M(I_2,I_3).
\]
\item
If $I_i$ is nontrivial for $i=1,2,3,$ then one of the following holds.
    \begin{enumerate}
    \item\label{lmNonTrivEithOrcase1}
    $M(I_1,I_3)\subset M(I_1,I_2)\cup M(I_2,I_3),$
    \item\label{lmNonTrivEithOrcase2}
    $M(I_1,I_3)\subset M(I_1,\overline{I_1})$ and $M(I_2,I_3)=\overline{M(I_2,I_3)}$.
    \item\label{lmNonTrivEithOrcase3}
    $M(I_1,I_3)\subset M(I_3,\overline{I_3})$ and $M(I_1,I_2)=\overline{M(I_1,I_2)}.$
    \end{enumerate}
\end{enumerate}
 \end{lm}
\begin{proof}
\begin{enumerate}
\item
This is straightforward set theory.
\item
By Theorem~\ref{lmRelChar} there is a simple closed geodesic $\gamma$ such that
\[
I_i=\left\{z\in\mathcal{C}(\gamma)|\rho_{0,i}<\rho(z)<\rho_{1,i}\right\},
\]
for $i=1,2,3$. Suppose without loss of generality that
\[
\rho_{0,1}\leq\rho_{0,3}.
\]
Assume
\begin{equation}\label{lmlmNonTrivEithOrAss}
M(I_1,I_3)\not\subset M(I_1,I_2)\cup M(I_2,I_3).
\end{equation}
Considering the definition of $M(\cdot,\cdot)$, this assumption implies that $\gamma\subset\partial\Sigma$ and that
\begin{equation}\label{eqBothSides}
\rho_{0,1}\rho_{1,3}<0.
\end{equation}
We claim, further, that
\begin{equation}\label{OrderOfIi}
\rho_{0,1}\leq\rho_{0,2}<\rho_{1,2}\leq\rho_{1,3}.
\end{equation}
Indeed, if $\rho_{1,2}>\rho_{1,3}$ then combining the definition of $M(\cdot,\cdot)$ and inequality~\eqref{eqBothSides} it would follow that $M(I_1,I_3)\subset M(I_1,I_2)$. Similarly, if $\rho_{0,2}<\rho_{0,1}$ we would get that $M(I_1,I_3)\subset M(I_2,I_3)$. Now, if $|\rho_{0,1}|\geq|\rho_{1,3}|$, inequality~\eqref{OrderOfIi} implies condition~\ref{lmNonTrivEithOrcase2}. Indeed, the first part of condition~\ref{lmNonTrivEithOrcase2} is immediate. Furthermore, we must have in this case $\rho_{0,2}\rho_{1,3}<0$ for otherwise we would have $M(I_1,I_2)=M(I_1,\overline{I_1})$. By the first part of condition~\ref{lmNonTrivEithOrcase2} this would contradict equation~\eqref{lmlmNonTrivEithOrAss}. Thus $M(I_2,I_3)$ meets $\partial\Sigma$. Since $M(I_2,I_3)$ is clean, this implies $M(I_2,I_3)$ is conjugation invariant. If $|\rho_{0,1}|\leq|\rho_{1,3}|$ we get condition~\ref{lmNonTrivEithOrcase3} by switching the roles of $I_1$ and $I_3$.
\end{enumerate}
\end{proof}
\begin{lm}\label{ConjInvTrans}
Suppose $I_1$ and $I_2$ are topologically related. Then $I_1$ is conjugation invariant if and only if $I_2$ is.
\end{lm}
\begin{proof}
If $I_1$ is conjugation invariant then by Lemmas~\ref{lmI1sharpI2} and~\ref{lmCleanConjEmb} it follows that $M(I_1,I_2)$ is conjugation invariant. Again applying Lemma \ref{lmCleanConjEmb} it follows that $I_2$ is conjugation invariant. The converse is obtained by exchanging the roles of $I_1$ and $I_2$.
\end{proof}
\begin{lm}\label{lmRelEq}
For $i=1,2,3$ let $I_i$ be admissible annuli. Write $I_i=B_i\setminus B'_i$ where $B'_i$ is a clean closed disc, $B_i$ a clean open disc and $B'_i$ is concentric with $B_i$. Suppose that the pairs $(I_1,I_2)$ and $(I_2,I_3)$ are topologically related. If $B'_1\nsubseteq M(I_2,I_3)$ and $B'_3\nsubseteq M(I_1,I_2)$ then $I_1$ and $I_3$ are topologically related.
\end{lm}
\begin{proof}
We show first that $I_1\cup I_3$ is clean. If $I_1$ is conjugation invariant then by Lemma~\ref{ConjInvTrans} so is $I_3$. If $I_1$ is not conjugation invariant then there is a component $A$ of $\tilde{\Sigma}\setminus\partial\Sigma$ which contains $I_1$. By Theorem~\ref{lmRelChar}\ref{lmRelChar2}, $B'_2\cap B'_1\neq\emptyset$. Since $I_2$ is not conjugation invariant it follows that $I_2\subset A$. By repeating this argument, $I_3\subset A$. Thus $I_1\cup I_3$ is clean.

We now wish show that $B'_1\cap B'_3\neq\emptyset$. For this we show first that either $B'_1\subset B_2\cup B_3$ or $B'_3\subset B_1\cup B_2$. By Theorem \ref{lmRelChar}, $B'_1\cap B'_2\neq\emptyset$ and $B'_2\cap B'_3\neq\emptyset$. That is,
\[
d(p_1,p_2)\leq r_1'+r_2'
\]
and
\[
d(p_2,p_3)\leq r_2'+r_3'.
\]
Let now $x\in B'_1.$ Then
\[
d(x,p_2)\leq d(x,p_1)+d(p_1,p_2)\leq 2r'_1+r'_2,
\]
and
\begin{equation}\label{lmRelEqEstrt30}
d(x,p_3)\leq d(x,p_1)+d(p_1,p_2)+d(p_2,p_3)\leq 2r'_1+2r'_2+r'_3.
\end{equation}
Suppose $x\not\in B_2\cup B_3$ then $2r'_1+r'_2> r_2\geq5r'_2$. This implies
\begin{equation}\label{lmRelEqEstrt1}
r'_2<\frac1{2}r'_1.
\end{equation}
Combining estimates~\eqref{lmRelEqEstrt30} and~\eqref{lmRelEqEstrt1} and the estimate $r_3\geq 5r_3'$, we get
\begin{equation}\label{lmRelEqEstrt2}
r'_3<\frac{3}{4}r'_1.
\end{equation}
On the other hand, for \textit{any} $x\in B'_3$ we have
\[
d(x,p_1)\leq d(x,p_3)+d(p_1,p_3)\leq r'_3+d(p_1,p_3).
\]
Therefore, combining estimates \eqref{lmRelEqEstrt1} and \eqref{lmRelEqEstrt2},
\[
d(p_1,p_3)\leq r'_1+2r'_2+r'_3<(2+\frac{3}{4})r'_1< r_1-r'_3.
\]
That is, $B'_3\subset B_1\subset B_1\cup B_2$ as claimed.

We use this to show that $B'_1\cap B'_2\cap B'_3\neq\emptyset$. Suppose by contradiction
\begin{equation}\label{lmRelEqContAs}
B'_1\cap B'_2\cap B'_3=\emptyset.
\end{equation}
Then in case $B'_1\subset B_2\cup B_3$, assumption \eqref{lmRelEqContAs} implies $B'_1\subset M(I_2,I_3)$. Similarly in case $B'_3\subset B_1\cup B_2$, assumption  \eqref{lmRelEqContAs} implies $B'_3\subset M(I_1,I_2)$. In any case we get a contradiction to the assumptions of the lemma.
\end{proof}

\subsection{Essential disjointeness}
\begin{tm}\label{tmEssDisj}
There is a constant $K_1$ with the following significance. Let $I_1,I_2\in\hat{\mathcal{A}}_h$. Suppose $C(K_1,K_1;I_1)\cap C(K_1,K_1;I_2)\neq\emptyset.$ Then either $\tilde{\Sigma}$ is a torus covered by $I_1$ and $I_2$, or
\[
b_1(I_1\cup I_2)\leq 1.
\]
Here for a topological space $X$, $b_1(X)$ denotes the first Betti number of $X$.
\end{tm}
The proof of Theorem~\ref{tmEssDisj} spans this subsection.
\begin{df}\label{dfEsDisj}
$I_1$ and $I_2$ are said to be \textbf{essentially disjoint} if
\[
C(K_1,K_1;I_1)\cap C(K_1,K_1;I_2)=\emptyset,
\]
where $K_1$ is a constant satisfying Theorem \ref{tmEssDisj} that is fixed once and for all. In later uses it will be convenient to assume further that $K_1\geq c_2+\pi$.
\end{df}

\begin{lm}\label{EsDisjCyl}
Let $I_1$ and $I_2$ be admissible nontrivial annuli. If
\[
I_1\cap I_2\neq\emptyset
\]
then each component\footnote{The possibility of more than one component appears when $genus(\tilde{\Sigma})=1.$}  of $I_1\cap I_2$ is a sub-cylinder of $I_i$ for $i=1,2$.
\end{lm}
\begin{proof}
By definition \ref {dfAdmAnn} there are simple closed geodesics $\gamma_i$ such that $I_i$ is a sub-cylinder of $\mathcal{C}(\gamma_i)$ for $i=1,2$. First assume $genus(\tilde{\Sigma})>1$. The assumption
\[
I_1\cap I_2\neq\emptyset
\]
and Theorem \ref{TmThTh}\ref{it:coll} imply that $\gamma_1\cap\gamma_2\neq\emptyset$. If $\gamma_1$ intersects $\gamma_2$ transversally in a nonempty set then by \cite[4.1.1]{Bu} there is an $i\in\{1,2\}$ such that $\ell(\gamma_i)\geq 2\sinh^{-1}(1)$. This is contrary to the definition of admissibility. Since $\gamma_1$ and $\gamma_2$ are geodesics which intersect non transversally, $\gamma_1=\gamma_2$. The intersection of sub-cylinders of a given cylinder is a sub-cylinder. Thus the claim follows.

Assume now that $genus(\tilde{\Sigma})=0$. Then by definition $\partial\Sigma\neq\emptyset$ and $I_i$ are both sub-cylinders of the $\mathcal{C}(\partial\Sigma)$. So, the claim follows as before. Finally, assume $genus(\tilde{\Sigma})=1$. We claim that $\gamma_1$ is parallel to $\gamma_2$. Indeed, the alternative is that $\gamma_1$ intersects $\gamma_2$ transversally. But then
\[
\ell(\gamma_1)\ell(\gamma_2)>Area(\Sigma;h)=1>\sinh^{-1}(1),
\]
contradicting the admissibility of $I_1$ and $I_2$. This implies the claim.
\end{proof}

\begin{lm}\label{lmBettyNmuDiscAn}
Let $I$ and $B$ be a geodesic annulus and a geodesic disc, respectively, in the hyperbolic disc, in the Riemann sphere, or in the flat plane. Then $b_1(I\cup B)\leq 1$.
\end{lm}
\begin{proof}
Suppose by contradiction otherwise. Then, by the Meyer Vietoris sequence, $I\cap B$ has at least two components. In particular, there is a boundary component of $\gamma\subset\partial I$ such that $\gamma\cap\partial B$ consists of at least four points. On the other hand, any two geodesic circles are also circles with respect to the flat metric on the disc. Any two such circles intersect in at must two points. A contradiction.
\end{proof}
\begin{lm}\label{EsDisjAn}
There is a constant $K$ with the following significance. Let $I_1$ and $I_2$ be admissible trivial annuli. Assume that $b_1(I_1\cup I_2)\geq 2$ Then
\[
C(K,K;I_1)\cap C(K,K;I_2)=\emptyset.
\]
\end{lm}
\begin{proof}
Write $I_i=A(r_i,r'_i;p_i)$, $B_i=B_{r_i}(p_i)$, and $B'_i=B^c_{r'_i}(p_i)$ for $i=1,2$. First assume $B'_1\cap B'_2=\emptyset$. Note that the assumption on $I_1\cup I_2$ implies
\begin{align}\label{EqEsDisj2}
B'_i\not\subset B_{i\bmod{2} +1}.
\end{align}
Indeed, suppose for example that $B'_1\subset B_2\setminus B'_2$ then
\[
I_1\cup I_2=I_2\cup B_1.
\]
By Lemma \ref{lmBettyNmuDiscAn} this would imply $b_1(I_1\cup I_2)\leq 1$ contradicting the assumption.

It follows that $r_1< d(p_1,p_2)+r'_2$ and $r_2< d(p_1,p_2)+r'_1$. The combination of these inequalities with the condition $r'_i\leq\frac1{5}r_i$ in the definition of admissibility implies that
 \begin{align}\label{EqEsDisj3}
 d(p_1,p_2)> \frac{2}{5}(r_1+r_2).
 \end{align}
 Let now $s_i=\frac{2r_id(p_1,p_2)}{3(r_1+r_2)}$ for $i=1,2$. Write $J_i=A(s_i,r'_i;p_i)$. By equation~\eqref{EqEsDisj3},
 \[
 s_i>\frac1{5}r_i\geq r'_i.
 \]
In particular, $J_i\neq\emptyset$. We have $s_1+s_2<d(p_1,p_2)$, so
\begin{align}\label{EqEsDisj4}
J_1\cap J_2=\emptyset.
\end{align}
 It thus suffices to show that there is a universal constant $K$ such that $C(K,K;I_i)\subset J_i$ for $i=1,2$. Write $L_i:=A(r_i,s_i;p_i)$. $L_i$ is a sub-cylinder of $I_i$, so $I_i\backslash L_i=C(0,Mod(L_i);I_i)$. Note that $J_i=I_i\setminus L_i$, so $C(Mod(L_i),Mod(L_i);I)\subset J_i$. It therefore suffices to uniformly bound $Mod(L_i)$. We have
\begin{align}\label{EqEsDisjj5}
Mod(L_i)=\int_{s_i}^{r_i}\frac{dr}{h_{\theta}(r)}
\end{align}
We have either ${h_{\theta}(r)}={\sin(r)}$ and $r_i\leq\pi/2$, or ${h_{\theta}(r)}={r}$, or ${h_{\theta}(r)}={\sinh(r)}$, so we need only verify the boundedness of expression~\eqref{EqEsDisjj5} when $s_i\rightarrow 0$. But $s_i/r_i\geq\frac{4}{15}$, so this is obvious.

Now assume $B'_1\cap B'_2\neq\emptyset$. First note that the assumption on $I_1\cup I_2$ implies that  either $B_1'\not\subset B_2$ or $B_2'\not\subset B_1$. Indeed, otherwise
\[
I_1\cup I_2=B_1\cup B_2\setminus B'_1\cap B'_2.
\]
By Lemma \ref{lmRelDCon} we would then have that $b_1(I_1\cup I_2)=1$ in contradiction to the assumption of the Lemma. Thus we may, without loss of generality, assume $B_1'\not\subset B_2$. Since $B'_1\cap B'_2\neq\emptyset$, $d(p_1,p_2)\leq r_1'+r_2'$. On the other hand , since $B'_1\not\subset B_2$, $r_2\leq d(p_1,p_2)+r'_1.$ The combination of these two inequalities implies that $r_2\leq \frac{5}{2}r'_1$. We thus have
 \[
 d(p_1,p_2)+r_2\leq r_1'+r_2'+\frac{5}{2}r'_1\leq4r'_1.
  \]
 Therefore, letting $J=A(4r'_1,r'_1;p_1)$, we have $(I_1\setminus J)\cap I_2=\emptyset$. On the other hand
 \[
C(Mod(J),Mod(J);I_1)\subset I_1\setminus J.
\]
We have that $Mod(J)$ is bounded from above by some constant $K$ which is independent of $r'_1$. The claim follows.
\end{proof}

\begin{lm}\label{EsDisjAnCyl}
There is a constant $K$ with the following significance. Let $I_1=A(r_1,r'_1;p)\in\mathcal{A}_h$ be trivial and $I_2\in\hat{\mathcal{A}}_h$ be nontrivial. Suppose $I_1\not\subset I_2$. Then $I_1\cap C(K,K;I_2)=\emptyset$.
\end{lm}
\begin{proof}
First assume $genus(\tilde{\Sigma})\geq 1$. By definition, there is a simple closed geodesic $\gamma$ such that $I_2$ is a sub-cylinder in $\mathcal{C}(\gamma)$. Recall the definition of $(\rho, \theta)$ coordinates on $\mathcal{C}(\gamma)$. Write $\rho_0=\inf\{\rho(z)|z\in I_1\cap I_2\}$ and $\rho_1=\sup \{\rho(z)|z\in I_1\cap I_2\}$. Let $I\subset\mathcal{C}(\gamma)$ be given in $(\rho,\theta)$ coordinates by
\[
\{(\rho,\theta)|\rho_0\leq\rho\leq\rho_1\}.
\]
Denote by $\beta_i$ the components $\{\rho=\rho_i\}$ of $\partial I$ for $i=0,1$. Now note that since $B_{r_1}(p_1)$ is a disc of radius $r<\frac1{3}\inj(\tilde{\Sigma};p_1)$, we have
\begin{align}\label{eqRhoinj}
\rho_1-\rho_0<\sup_{z\in I}\inj(\tilde{\Sigma};z).
\end{align}
Indeed, this is obvious if $p_1\in I$. Otherwise, suppose without loss of generality that $\beta_0$ lies between $p_1$ and $\beta_1$, and let $p'$ be the intersection of the perpendicular from $p_1$ with $\beta_0$. Then $B_{\rho_1-\rho_0}(p')\subset B_r(p_1)$. In particular, $\rho_1-\rho_0<\inj(\tilde{\Sigma};p')$. This establishes inequality~\eqref{eqRhoinj}.

Now, $\inj(\tilde{\Sigma},\cdot)$ is either constant or has no local maximum in $I$. When $genus(\tilde{\Sigma})>1$ this can be seen from relation~\eqref{injEq}. Else $\inj$ is constant. We may therefore assume without loss of generality that $\inj(\tilde{\Sigma},\cdot)$ attains its supremum at $\rho_0$ (we no longer make the assumption from the previous paragraph about $\beta_0$). By the assumption on the genus, $\beta_0$ is not contractible. Therefore,
\[
\inj(\Sigma;(\rho_0,\theta(\cdot)))\leq \frac12\ell(\beta_0)=\pi h_{\theta}(\rho_0).
\]
Thus, we have the estimate
\[
Mod (I)\leq\int_{\rho_0}^{\rho_0+\pi h_{\theta}(\rho_0)}\frac{d\rho}{h_{\theta}(\rho)}.
\]
The last expression is bounded by a universal constant $K$. Indeed, in case $genus(\tilde{\Sigma})=1$, $h_{\theta}$ is constant and the bound is obvious.
Otherwise, using Theorem~\ref{TmThTh}\ref{it:spco}, the last expression is estimated by $Ce^{\ell(\gamma)\pi\cosh(\rho_0)}$ for an a priori constant $C$. Using the definition of $\rho_0$ and $\mathcal{C}(\gamma)$ we have
\[
\ell(\gamma)\pi\cosh(\rho_0)\leq \ell(\gamma)\pi\cosh(w(\gamma))\leq C'.
\]
Here $w(\gamma)$ is as defined in Theorem~\ref{TmThTh}\ref{it:coll} and $C'$ is an a priori constant.

When $genus(\tilde{\Sigma})=0$, $\mathcal{C}(\gamma)$ is an annulus in $\tilde{\Sigma}$ and the claim follows with slight modification in the same way  as Lemma \ref{EsDisjAn}.
\end{proof}

\begin{proof}[Proof of Theorem \ref{tmEssDisj}]
Let $K_1$ be a constant as in Lemmas \ref{EsDisjAn} and \ref{EsDisjAnCyl}. Suppose $C(K_1,K_1;I_1)\cap C(K_1,K_1;I_2)\neq\emptyset$. If $I_1$ and $I_2$ are both trivial annuli, the theorem is just a restatement of Lemma \ref{EsDisjAn}. If both $I_1$ and $I_2$ are nontrivial and $\Sigma_\C$ is not a torus covered by $I_1$ and $I_2$, Lemma \ref{EsDisjCyl} implies $I_1$ and $I_2$ intersect in a sub-cylinder $I$. So, $I_1\cup I_2$ is a sub-cylinder of $\mathcal{C}(\gamma)$ for some simple closed geodesic $\gamma$. In particular $b_1(I_1\cup I_2)=1$. Finally, if $I_1$ is trivial and $I_2$ is nontrivial. Then by Lemma~\ref{EsDisjAnCyl}, $I_1\cup I_2=I_2$.
\end{proof}

\subsection{Long annuli}
\begin{lm}\label{lmExLNdisj}
There is a constant $K_2$ with the following significance. Let $L\geq K_2$ and let $I_1,I_2\in\mathcal{A}_h$ be essentially disjoint. Suppose $Mod(I_i)>4L$ for $i=1,2$. Suppose further that $I_1$ and $I_2$ are topologically related. Then
\begin{align}\label{eqlmExLNdisj}
Mod (m(I_1,I_2))>\max\{Mod(I_1),Mod(I_2)\}+2L.
\end{align}
\end{lm}
\begin{proof}
When $I_1$ and $I_2$ are cylinders, it is straightforward to verify that the claim holds whenever $K_2\geq K_1$ where $K_1$ is the constant from Definition \ref{dfEsDisj}. Otherwise, for $i=1,2$, write $I_i= A(r_i,r'_i,p_i)$ and $B'_i=B^c_{r'_i}(p_i)$. Assume without loss of generality that $r'_2\leq r'_1$. Then Theorem~\ref{lmRelChar}\ref{lmRelChar2} implies
\[
d(p_2,\partial B_1)\geq r_1-2r'_1,
\]
and admissibility implies $p_2\in B_1$. Let $r=d(p_2,\partial B_1)$ and write
\[
I:=m(I_1,I_2)=B_r(p_2)\backslash B'_2.
\]
Let
\[
\Delta=\int_{r_1-2r_1'}^{r_1}\frac{dr}{h_{\theta}(r)}.
\]
Then, applying equation~\eqref{ModEq},
\begin{align}
Mod(I)&=\int_{r'_2}^{r}\frac{dr}{h_{\theta}(r)}\notag\\
&\geq\int_{r_2'}^{r_1-2r_1'}\frac{dr}{h_{\theta}(r)}\notag\\
&=\int_{r'_2}^{r_1}\frac{dr}{h_{\theta}(r)}-\Delta\notag\\
&\geq Mod(I_1)+Mod(I_2)-2K_1-\Delta\notag.
\end{align}
The claim of the lemma will follow if we find a $K_2$ such that
\[
\Delta+2K_1\leq 2K_2.
\]
In other words it suffices to bound $\Delta$ uniformly from above. By the restrictions on the range of $r_1$ in the definition of admissibility, $h_{\theta}$ is monotone increasing. See equation~\eqref{eq:pch}. Therefore,
\[
\Delta\leq \frac{2r'_1}{h_{\theta}(r_1-2r'_1)}\leq\frac{r'_1}{h_{\theta}(3r'_1)}.
\]
But whether the curvature of $h$ is positive, negative or vanishing,
\[
\lim_{r\rightarrow 0} \frac{r}{h_{\theta}(r)}=1.
\]
Relying again on equation~\eqref{eq:pch} we get that $\Delta$ is uniformly bounded from above whenever the curvature is non-positive. When the curvature is positive, we still have that  $\Delta$ is uniformly bounded in the range of admissibility
\[
r'_1\in(0,\frac{\pi}{15}].
\]
\end{proof}

\begin{lm}\label{lmI1Isharp2I}
Let $I_i=A(r_i,r_i';p_i)\in\mathcal{A}_h$ for $i=1,2$. Suppose $r_2\leq r_1$, $I_1$ is topologically related to $I_2$, and $M(I_1,I_2)\neq I_1\cup I_2$. Then
\begin{enumerate}
\item\label{EqRtagEst}
\begin{align}
d(p_1,p_2)\leq r'_1+r'_2,\notag
\end{align}
\item\label{Eqr3r1tagest}
\begin{equation}
r_2<\frac{5}{2}r_1',\notag
\end{equation}
\item\label{eqJII1}
\[
M(I_1,I_2)= I_1\cup m(I_1,I_2).
\]
\end{enumerate}

\end{lm}
\begin{proof}
Write $B_i=B_{r_i}(p_i)$, $B'_i=B^c_{r'_i}(p_i)$ and $J=M(I_1,I_2)$. By Lemma \ref{lmRelEq},
\begin{align}\label{eqbtag1btag2}
B'_1\cap B'_2\neq\emptyset.
\end{align}
Part~\ref{EqRtagEst} is an immediate consequence.
To prove part~\ref{Eqr3r1tagest}, verify using $J=B_1\cup B_2\setminus \left(B'_1\cap B'_2\right)$ that
\begin{align}\label{seteq}
J\setminus (I_1\cup I_2)\subset \left(B'_1\setminus B_2\right)\cup \left(B'_2\setminus B_1\right).
\end{align}
On the other hand, relation~\eqref{eqbtag1btag2}, admissibility, and the assumption $r_2\leq r_1$, imply $B'_2\subset B_1$. Indeed, we have for any $q\in B'_2$
\[
d(q,p_1)\leq 2r'_1+r'_2\leq \frac25r_1+\frac15r_2\leq\frac35 r_1<r_1.
\]
So, by relation~\eqref{seteq}, $B'_1\setminus B_2\neq\emptyset$. Let $q\in B'_1\setminus B_2$. Then $d(q,p_1)<r'_1$ and $r_2<d(q,p_2)$. Combining these inequalities we get
\begin{equation}\label{lmI1Isharp2IEq111}
r_2<d(q,p_2)\leq d(q,p_1)+d(p_1,p_2)<2r'_1+r'_2.
\end{equation}
Since $I_2$ is admissible, $r'_2\leq\frac1{5}r_2$. Thus, inequality~\eqref{lmI1Isharp2IEq111} implies part~\ref{Eqr3r1tagest}.

We prove part~\ref{eqJII1}. By definition, $I_1$ and $m(I_1,I_2)$ are subsets of $M(I_1,I_2)$. We prove the reverse inclusion. Using part~\ref{Eqr3r1tagest} and admissibility, one verifies that $B_2\subset B_1$. Therefore,
 \[
 M(I_1,I_2)\setminus I_1=B_1\setminus \left((B'_1\cap B'_2)\cup I_1\right)=B'_1\setminus B'_2.
\]
Write $r=d(p_2,\partial B_1)$. We need to show that
\[
B'_1\setminus B'_2\subset m(I_1,I_2)=B_r(p_2)\setminus B'_2.
\]
For this it suffices to show that $B'_1\subset  B_{r}(p_2)$. That is,
\[
d(p_1,p_2)+r_1'< r.
\]
But by parts~\ref{EqRtagEst} and~\ref{Eqr3r1tagest} we get
\begin{align}
r&=d(p_2,\partial B_1)\notag\\& \geq d(p_1,\partial B_1)-d(p_1,p_2)\notag\\
&\geq r_1-r'_1-r'_2\notag\\&\geq3r'_1\notag\\&>d(p_1,p_2)+r_1'\notag.
\end{align}
\end{proof}
\begin{lm}\label{lmK3def}
There is a constant $K_3$ with the following significance. Let  $I_i=A(r_i,r'_i;p_i)\subset \mathcal{A}_h$, for $i=1,2$. Let $L\geq K_3$ and suppose
\[
Mod I_i\geq2L.
\]
Suppose $I_1$ is topologically related to $I_2$ and let $I=m(I_1,I_2)$. Then
\begin{enumerate}
\item\label{lmK3defit1} $M(I_1,I_2)=I_1\cup I_2\cup C(L,L;I)$.
\item\label{lmK3defit2} $I\setminus C(L,L;I)\subset I_1\cup I_2.$
\end{enumerate}
\end{lm}
\begin{proof}
Write $B_i=B_{r_i}(p_i)$, $B'_i=B^c_{r'_i}(p_i)$, and $J=M(I_1,I_2)$. If
\[
J=I_1\cup I_2
\]
there is nothing to prove. We thus assume $J\neq I_1\cup I_2$. We first show that for some fixed $K_3$ chosen large enough, part~\ref{lmK3defit2} holds. Let $J_1=I\setminus C(0,L;I)$ and $J_2=I\setminus C(L,0;I)$. Assume without loss of generality that $r_2\leq r_1$. Then $I$ is centered at $p_2$. Since, furthermore, $Mod (I_2)>L$, it follows that $J_2\subset I_2$\footnote{See remark~\ref{COrientConv}.}. It remains to show that $J_1\subset I_1$. Let $r=d(p_2,\partial B_1)$. There is a real number $r'$ such that $J_1=A(r,r',p_2)$. Clearly, $B_r(p_2)\subset B_1$. Therefore, to show the inclusion $J_1\subset I_1$ it suffices to show that $B^c_{r'_1}(p_1)\subset B^c_{r'}(p_2)$. That is, it suffices to show that
\[
r'-r'_1>d(p_1,p_2).
\]
By parts~\ref{EqRtagEst} and~\ref{Eqr3r1tagest} of Lemma~\ref{lmI1Isharp2I}, it suffices that $r'>3r'_1$. Considering the definition of $r'$, this is equivalent to the claim
\[
Mod(A(r,3r'_1;p_2))> L.
\]
Let $L':=Mod(A(r,3r'_1;p_2)).$ Since $B'_1\cap B'_2\neq\emptyset,$ it is clear that
\[
r\geq r_1-2r'_1.
\]
So,
\begin{align}
L'&\geq Mod\left(A(r_1-2r'_1,3r'_1;p_2)\right)\notag\\
&=Mod\left(A(r_1-2r'_1,3r'_1;p_1)\right)\notag\\
&>Mod\left(A\left(\frac1{2}r_1,3r'_1;p_1\right)\right)\notag\\
&=Mod(I_1)-Mod\left(A\left(r_1,\frac1{2}r_1;p_1\right)\right)-Mod\left(A(3r'_1,r'_1;p_1)\right)\notag\\
&\geq  2L-\int_{\pi/6}^{\pi/3}\frac{dr}{\sin r}-\int_{\pi/9}^{\pi/3}\frac{dr}{\sin r}\notag
\end{align}
For the last line, see equations~\eqref{ModEq},~\eqref{eq:pch} and the definition of admissibility. Choosing
\[
K_3=\int_{\pi/6}^{\pi/3}\frac{dr}{\sin r}+\int_{\pi/9}^{\pi/3}\frac{dr}{\sin r},
\]
we thus get $J_1\subset I_1.$ This establishes part \ref{lmK3defit2}. We prove part~\ref{lmK3defit1}. The inclusion
 \[
 M(I_1,I_2)\supset I_1\cup I_2\cup C(L,L;I),
 \]
follows from definitions. The reverse inclusion is an immediate consequence of Lemma~\ref{lmI1Isharp2I}\ref{eqJII1} and part~\ref{lmK3defit2}.

\end{proof}
\section{Construction of bubble decomposition}
Let $L_0>\max\{c_2,\log3/c_3\}$. It follows from the cylinder inequality that for any $(\Sigma,\mu)\in\mathcal{M}$ and any clean $I\subset\Sigma_{\mathbb{C}}$ with $\mu(I)\leq\delta_2$ and $Mod(I)>2L_0$, we have
\begin{align}\label{eqAnEnPart}
 \mu\left(C(L_0,L_0;I)\right)\leq\frac{\mu(I)}{3}.
 \end{align}
Let $L_0$ satisfy, further, $L_0>\max\{K_1,K_2,K_3\}$. Here, $K_1,$ $K_2$ and $K_3$ are the constants from Definition \ref{dfEsDisj}, Lemma \ref{lmExLNdisj} and Lemma \ref{lmK3def}, respectively.

\begin{df}
A \textbf{neck} is an $I\in\mathcal{A}_h$ with the property that $\tilde{\Sigma}\setminus I$ is $\mu$-stable. Write $L_1=4L_0$. A \textbf{long neck} is a neck $I$ which satisfies $\mu(I)\leq\delta_2/6$ and $Mod(I)\geq L_1$. \end{df}

When $genus(\tilde{\Sigma})\neq 1$, let $LN$ denote the set of long necks. Otherwise, recall the definition of $I_0$ appearing in the definition of $\alpha_0.$ Define $LN$ to be the set of long necks contained in $\Sigma\setminus I_0$. If $Mod(I_0)\geq L_1$, let $\tilde{LN}:=LN\cup\{I_0\}$. In all the other cases, whatever the genus, define $\tilde{LN}:=LN$.

\begin{df}\label{defMaxBubDecom}
A \textbf{maximal $\mu$-decomposition} is a bubble decomposition $\mathcal{B}$ with the following properties.
\begin{enumerate}
\item\label{tmMaxBubDecom1}
There exists a set $\tilde{\mathcal{B}}$ of pairwise essentially disjoint elements of $ \tilde{LN}$ such that
\[
\mathcal{B}:=\{C(2K_1,2K_1;I)|I\in\tilde{\mathcal{B}}\}.
\]
\item\label{tmMaxBubDecom2}
For any $v\in V_{{\mathcal{B}}}$, $\Sigma_v$ is $\mu$-stable.
\item\label{tmMaxBubDecom3}
For any $v\in V_{{\mathcal{B}}}$, $\Sigma_v$ contains no long necks.
\end{enumerate}
\end{df}

\begin{tm}\label{tmMaxBubDecom}
For any $(\Sigma,\mu)\in\mathcal{M}$ satisfying Assumption~\ref{as1}, there exists a maximal $\mu$-decomposition.
\end{tm}

To prove Theorem \ref{tmMaxBubDecom}, we define a relation $\sim$ on $LN$ as follows. For $I_1,I_2\in LN$, $I_1\sim I_2$ if and only if $I_1$ and $I_2$ are topologically related and $\mu(M(I_1,I_2))\leq\delta_2/2$.
\begin{lm}\label{lmIntRDisj}
Let $I_1\in LN$ be trivial. Write $I_1=B_1\setminus B_1$ for appropriate concentric discs in $\tilde{\Sigma}$. Let $I_2\in LN$. Then $B'_1\not\subset I_2$. Furthermore, if $I_3\in LN$ and $I_2\sim I_3$ then $B'_1\not\subset M(I_2,I_3)$.
\end{lm}
\begin{proof}
The component $B'_1$  of $\tilde{\Sigma}\setminus I_1$ is $\mu$-stable. That is,
\[
\mu(B'_1)\geq\delta_1/2>\delta_2/6.
\]
On the other hand, $\mu(I_2)\leq\delta_2/6$, and, by $\sim$-equivalence,
\[
\mu(M(I_2,I_3))\leq\delta_2/2<\delta_1/2.
\]
Both parts of the claim follow.
\end{proof}

\begin{lm}\label{lmSimEq}
The relation $\sim$ is an equivalence relation.
\end{lm}
\begin{proof}
Symmetry and reflexivity are obvious, so we need only establish transitivity. Let $I_i\in LN$  for $i=1,2,3$, and suppose $I_1\sim I_2$ and $I_2\sim I_3$. It follows from Theorem~\ref{lmRelChar}\ref{lmRelChar1} that either the three annuli are all trivial or all nontrivial. Suppose all are non trivial. Observe, using \ref{lmRelChar}\ref{lmRelChar1}, that $I_1$ and $I_3$ are topologically related.

Let now $J=M(I_1,I_3)$. We show first that
\begin{equation}\label{EqJdelta2Est}
\mu(J)\leq \delta_2.
\end{equation}
Let $I'=M(I_1,I_2)$ and $I''=M(I_2,I_3)$. If $J\subset I'\cup I''$ we have
\[
\mu(J)\leq \mu(I')+\mu(I'')\leq\delta_2/2+\delta_2/2=\delta_2.
\]
Otherwise, by Lemma \ref{lmNonTrivEithOr} we may assume without loss of generality that $J\subset M(I_1,\overline{I}_1)$ and $I''=\overline{I''}$. Let $J'=J\cap\Sigma$ and $J''=\overline{J'}$. It is easy to verify that $J'\subset I'\cup I''$. So, $\mu(J')\leq\delta_2$. Similarly, one verifies that
\[
J'\setminus C(L_0,L_0;J')\subset I_1\cup C(L_0,L_0;I'').
\]
Applying inequality \eqref{eqAnEnPart} we get
\[
\mu(J')\leq \frac{3}{2}(\delta_2/6+\delta_2/6)=\delta_2/2.
\]
Similarly, $\mu(J'')\leq\delta_2/2$. Inequality~\eqref{EqJdelta2Est} follows.

Applying inequality \eqref{eqAnEnPart} again,
\begin{align}\label{eqsimEquivLNEnEst}
\mu(J)&\leq \frac{3}{2}\mu(J\setminus C(L_0,L_0;J)).
\end{align}
Now note that by definition of $J$, $J\setminus C(L_0,L_0;J)$ is contained within one the following sets: $I_1\cup I_3$, $I_1\cup\overline{I}_1$, or $I_3\cup\overline{I_3}$. But
\[
\mu(I_i)=\mu(\overline{I}_i)\leq\frac{\delta_2}{6}
\]
 for $i=1,3$. Thus in any case we get that
 \[
 \mu(J\setminus C(L_0,L_0;J))\leq\frac{\delta_2}{3}.
 \]
Therefore by inequality~\eqref{eqsimEquivLNEnEst}
\[
\mu(J)\leq\frac{\delta_2}{2}
\]
as was to be proven.

Let now $I_i$ all be trivial. Write $I_i$ =$B_{i}\setminus B'_{i}$ where $B_i=B_{r_i}(p_i)$, $B'_i=B^c_{r'_i}(p_i)$ for some $p_i\in\tilde{\Sigma}$, $r'_i<r_i\in(0,\infty)$, and $i=1,2,3$. By Lemmas \ref{lmRelEq} and \ref{lmIntRDisj}, $I_1$ and $I_3$ are topologically related. Let $J=M(I_1,I_3)$. We need to show that $\mu(J)\leq\delta_2/2$. By Lemma \ref{lmK3def},
\[
J=I_1\cup I_3\cup C(L_0,L_0;I)
\]
where $I=m(I_1,I_3)$. Thus, $\mu(J)\leq\delta_2/3+\mu(C(L_0,L_0;I))$. To finish the proof we need to show that $\mu(C(L_0,L_0;I))\leq\delta_2/6$. Then
\[
I\subset J\subset M(I_1,I_2)\cup M(I_2,I_3).
\]
So, $\mu(I)\leq\delta_2$. On the other hand, by Lemma \ref{lmK3def},
\[
I\setminus C(L_0,L_0;I)\subset I_1\cup I_3.
\]
 Therefore, by definition of $L_0$,
\[
\mu(C(L_0,L_0;I))\leq\frac1{3}(\mu(I_1\cup I_3)+\mu(C(L_0,L_0;I))).
\]
So,
\[
\mu(C(L_0,L_0;I))\leq\frac12\mu(I_1\cup I_3)\leq\frac{\delta_2}{6}
\]
as required.
\end{proof}

\begin{lm}
For every $\sim$-equivalence class $c$ there is an $I\in c$ such that $Mod(I)$ is maximal in $c$ .
\end{lm}
\begin{proof}
Write $LN_0$ for the set of trivial long necks, and $LN_1$ for the set of nontrivial long necks. Let $R:=\max_{z\in\Sigma} \inj(\Sigma,z;h)$.
To give a trivial element of $\mathcal{A}_h$ is to give a point and two real numbers $r_1,r_2$ subject to some restrictions. This induces on $LN_0$ the topology of a subset of the compact bordered manifold
\[
X_0=\Sigma\times\left[0,\frac1{3}R\right]\times\left[0,\frac1{10}R\right].
\]
To give a nontrivial element of $\mathcal{A}_h$ is to give a simple closed geodesic and two real numbers subject to some restrictions. Thus, when
\[
genus(\tilde{\Sigma})>1,
\]
$LN_1$ can be assigned the topology of a subset of the compact bordered manifold
\[
X_1=\bigcup_{\{\gamma|\ell(\gamma)<\sinh^{-1}(1)\}}\left[-\frac1{2}Mod(\mathcal{C}(\gamma)),\frac1{2}Mod(\mathcal{C}(\gamma))\right]^2.
\]
$X_1$ is indeed compact since the number of simple closed geodesics $\gamma$ for which $\ell(\gamma)<\sinh^{-1}(1)$ is finite. See \cite{Bu}.  When $genus(\tilde{\Sigma})=0$ we have that $LN_1$ can be thought of as  a subset of
\[
[-\pi,\pi]^2.
\]
Finally, when $genus(\tilde{\Sigma})=1$, $LN_1$ is a subset of
\[
X_1=\left[-\frac1{2}Mod(\mathcal{C}(\alpha_0)),\frac1{2}Mod(\mathcal{C}(\alpha_0))\right].
\]
We show that $LN_i$ is closed in $X_i$. The conditions of stability, length and cleanness are closed conditions. However, admissibility alone is not a closed condition for trivial annuli because the inner radius of a trivial annulus must be positive. For nontrivial annuli it is not closed when $genus(\tilde{\Sigma})=0$, since $\mathcal{C}(\partial\Sigma)$ is not closed in this case. We show that the intersection of the set of admissible annuli with those having stable complement is closed.

First we show this for trivial annuli. The non-admissible points of $X_0$ in the closure of the trivial admissible annuli are  points of the form $(p,r,0)$. That is, annuli with internal radius $0$. Let
\[
r=\inf\left\{r'\in\left[0,\frac1{10}R\right]\Big|p\in\Sigma,\mu(B_{r'}(p))\geq\delta_1/2\right\}.
\]
Since $\Sigma$ is compact, $\frac{d\mu}{d\nu_h}$ is bounded. So, $r>0$. Thus, any trivial element of $LN_0$ has internal radius no less than $r$. The claim follows. For nontrivial annuli the claim follows in a similar manner.

Now we need to show that on $LN,$ the condition of equivalence is closed. The only non trivial point is to show that topological relatedness is a closed condition. By Theorem~\ref{lmRelChar} it suffices to show that there is an $a>0$ such that for any two equivalent trivial long necks of the form $I_i=B_{r_i}(p_i)\setminus B^c_{r'_i}(p_i)$, $i=1,2$, we have
\begin{equation}\label{eqr1r2intest}
A:=Area(B^c_{r'_1}(p_1)\cap B^c_{r'_2}(p_2);h_{can})\geq a.
\end{equation}
Write $B_i=B_{r_i}(p_i)$ and $B'_i=B^c_{r'_i}(p_i)$. We have $\mu(B'_i)\geq\delta_1/2$ and
\[
\mu(B'_1\cup B'_2\setminus B'_1\cap B'_2)=\mu(M(I_1,I_2))\leq\delta_2/2.
\]
Therefore,
\[
\mu(B^c_{r'_1}(p_1)\cap B^c_{r'_2}(p_2))\geq \delta_1/2-\delta_2/2\geq\delta_2/2.
\]
Since $\Sigma$ is compact, $\frac{d\mu}{d\nu_h}$ is bounded on $\Sigma$ by some constant $d$. Clearly,
\[
\delta_2/2\leq \mu(B^c_{r'_1}(p_1)\cap B^c_{r'_2}(p_2))\leq Ad.
\]
Inequality \eqref{eqr1r2intest} follows.

Finally, equation~\eqref{ModEq} shows that $Mod:LN\to\mathbb{R}$ is continuous with respect to the topology on $LN.$
\end{proof}

\begin{lm}\label{lmExLNdisj2}
Let $I_i\in LN$ for $i=1,2$. Suppose $I_1\sim I_2$  and $I_1$ is essentially disjoint from $I_2$. Write $I=m(I_1,I_2)$. Then $C(L_0,L_0;I)$ is a long neck which is $\sim$-equivalent to each of the $I_i$.
\end{lm}

\begin{proof}
Relying on Lemma \ref{lmExLNdisj} one verifies that $I\in \mathcal{A}_h$ and, furthermore, that $Mod(I)>L_1$. We have
\[
I\subset M(I_1,I_2),
\]
so
\[
\mu(I)\leq\mu(M(I_1,I_2))\leq\delta_2/2.
\]
Therefore, $\mu(C(L_0,L_0;I))\leq\delta_2/6$. We show that each component of
\[
\tilde{\Sigma}\setminus C(L_0,L_0;I)
\]
is $\mu$-stable. Let $A$ be one such connected component.  If $I$ is an admissible annulus, then, by construction, $A$  contains a component of $\tilde{\Sigma}\setminus I_i$ for either $i=1$ or $i=2$. If $I$ is an admissible cylinder then if
\[
genus(\tilde{\Sigma})>1,
\]
stability is automatic. It is left to treat the exceptional cases. When $genus(\tilde{\Sigma})=0$ the claim follows as in the case of trivial annuli. When $genus(\tilde{\Sigma})=1,$ the complement of $C(L_0,L_0;I)$ consists of a single component and so the claim follows by Assumption~\ref{as1}. We thus showed that $I$ is a long neck. For the remaining part of the claim, $I$ and each of the $I_i$ are nontrivially embedded in $M(I_1,I_2)$ and so are topologically related. Furthermore, we have $M(I_i,I))=M(I_1,I_2)$, so $\mu(M(I_i,I))\leq\delta_2$. The claim follows.
\end{proof}

\begin{lm}\label{lmIntAnRel}
Let $I_1,I_2\in LN$ and suppose $b_1(I_1\cup I_2)=1$. Then $I_1\sim I_2$.
\end{lm}
\begin{proof}
Suppose first that $I_1$ and $I_2$ are both trivial. Let $I_i=A(r_i,r'_i;p_i)$. By the Mayer Vietoris sequence, the assumption implies that
\[
b_1(I_1\cap I_2)=1.
\]
From this it follows that $B^c_{r'_1}(p_1)\subset B_{r_2}(p_2)$ and $B^c_{r'_2}(p_2)\subset B_{r_1}(p_1)$. It easily follows that $I_1\cup I_2$ is clean. Since $I_1$ and $I_2$ are nontrivially embedded in $I_1\cup I_2$, they are topologically related. Furthermore, $M(I_1,I_2)=I_1\cup I_2$. In particular
\[
\mu\left(M(I_1,I_2)\right)\leq\delta_2/3\leq\delta_2/2.
\]
Suppose now that $I_1$ and $I_2$ are both non trivial. If $I_1\cup I_2$ is clean then it is straightforward that $M(I_1,I_2)=I_1\cup I_2$. Otherwise, $I_1$ or $I_2$ is conjugation invariant. Without loss of generality assume $I_1$ is conjugation invariant. Then $M(I_1,I_2)=I_1\cup I_2\cup\overline{I}_2$. In any case, $\mu(M(I_1,I_2))\leq\delta_2/2$.
\end{proof}

\begin{lm}\label{lmLNb1}
For any $I_1,I_2\in LN$, $b_1(I_1\cup I_2)>0$.
\end{lm}
\begin{proof}
If $I_1$ is nontrivial, this is immediate. Otherwise, the claim is a consequence of Lemma \ref{lmIntRDisj}.
\end{proof}
\begin{lm}\label{lmEsdDisjInEq}
Let $c$ be a $\sim$-equivalence class. Let $I_1\in c$ have maximal modulus. Let $I_2\in LN$. $I_1$ and $I_2$ are essentially disjoint if and only if $I_2\not\in c$.
\end{lm}
\begin{proof}
Suppose  $I_1$ and $I_2$ are essentially disjoint and suppose by contradiction $I_2\in c$. Write $I:=C(L_0,L_0;m(I_1,I_2))$. By Lemma \ref{lmExLNdisj2}, $I\in LN\cap c$. By Lemma \ref{lmExLNdisj}, $Mod(I)>Mod(I_1)$. This is a contradiction. Conversely, suppose $I_1$ is not essentially disjoint from $I_2$. Recall that we excluded the trivial case $\tilde{\Sigma}\neq I_1\cup I_2$. Therefore, combining Theorem \ref{tmEssDisj} and Lemma \ref{lmLNb1}, we have $b_1(I_1\cup I_2)=1$ as long as $\tilde{\Sigma}\neq I_1\cup I_2$, we conclude $I_1\sim I_2$.
\end{proof}

\begin{lm}\label{lmcConjInv}
Let $c$ be a $\sim$-equivalence class.
\begin{enumerate}
\item
The elements of $c$ are either all trivial or all nontrivial. In the first case we say that $c$ is trivial, in the second case we say that it is nontrivial.
\item
If $c$ is trivial and then either all elements of $c$ are conjugation invariant or there is a component $A$ of $\Sigma_{\C}\setminus \partial\Sigma$ such that they are all contained in $A$.
\item
If $c$ is nontrivial then either the maximal elements of $c$ are conjugation invariant or there is a component $A$ of $\Sigma_{\C}\setminus \partial\Sigma$ such that they are all contained in $A$.
\end{enumerate}
\end{lm}
\begin{proof}
\begin{enumerate}
\item
This follows from Theorem~\ref{lmRelChar} since $\sim$-equivalence entails topological relatedness.
\item
For any $I_1,I_2\in c$, $I_1$ and $I_2$ embed nontrivially in $M(I_1,I_2)$ which is clean and doubly connected. Since these annuli are all trivial, none of them contains a component of $\partial\Sigma$. It thus follows from Lemma \ref{lmCleanConjEmb} that $I_1$ and $I_2$ are either both conjugation invariant or both contained in the same component of $\Sigma_{\C}\setminus\partial\Sigma$.
\item
Suppose there is a maximal element $I_1$ of $c$ that is not conjugation invariant. Since $I_1$ is clean, there is a component $A$ of $\Sigma_{\C}\setminus \partial\Sigma$ such that $I_1\subset A$. Suppose by contradiction that there is an $I_2\in c$ such that $I_2\not\subset A$. Write $J=C(L_0,L_0;M(I_1,I_2))$. Then maximality of $I_1$ easily implies
\[
J=M(I_1,\overline{I}_1).
\]
But then we get the contradiction $J\in LN\cap c$ and
\[
Mod(J)>Mod(I_1).
\]
\end{enumerate}
\end{proof}

\begin{lm}\label{lmBubDecompmustab}
Let $\mathcal{B}$ be a bubble decomposition consisting of elements of $LN$. Let $I=\Sigma_v$ for some $v\in V_{\mathcal{B}}$ and suppose $b_1(I)=1$. Suppose $I$ is bordered by  $\sim$-inequivalent elements $I_1,I_2\in\mathcal{B}\subset LN$. Then $I$ is $\mu$-stable.
\end{lm}
\begin{proof}
First we claim that $I_1$ and $I_2$ are topologically related. To see this note first that $I_1$ and $I_2$ freely homotopic, so they are either both trivial or both non-trivial. If both are nontrivial, the claim follows from Lemma~\ref{lmCleanMC}. Suppose both are trivial. Write $I_i=B_i\setminus B'_i$ where $B_i$ and $B'_i$ are concentric discs. Clearly we may assume with no loss of generality that $B_2\subset B'_1$ and so, $I=B'_1\setminus B_2$. In particular $B'_1\cap B'_2\neq\emptyset$. Further, $I_1\cap I_2$ is clean. Indeed, the only alternative is that $B_1$ is conjugation invariant while $B_2$ is not, but in that case $I$ is not conjugation invariant. Since by definition $\mathcal{B}$ is conjugation invariant, this is a contradiction. By Theorem~\ref{lmRelChar}\ref{lmRelChar2}, $I_1$ is topologically related to $I_2$.

We now show that $I$ is $\mu$-stable. In case the $I_i$ are trivial,
\[
I=M(I_1,I_2)\setminus (I_1\cup I_2).
\]
So, by $\sim$-inequivalence, $\mu(I)>\delta_2/6$. Suppose the $I_i$ are nontrivial. If $I_1\cup I_2$ is clean, we have $I= M(I_1,I_2)\setminus (I_1\cup I_2)$ and the claim follows as before. Otherwise, without loss of generality $I_1$ is conjugation invariant while $I_2$ is not. Then
\[
M(I_1,I_2)= I_1\cup I \cup I_2\cup \overline{I\cup I_2}.
\]
Suppose by contradiction that $\mu(I)<\delta_2/6$. Write $M=M(I_1,I_2)$. Then $\mu(M)<\delta_2$. Let $N=M\setminus C(L_0,L_0;M)$. We have $N\subset I_2\cup\overline{I}_2$. In particular $\mu(N)\leq\delta_2/3$. By definition of $L_0$ it follows that
\[
\mu(M)<\delta_2/2.
\]
That is, $I_1\sim I_2$. A contradiction.
\end{proof}
\begin{proof}[Proof of Theorem \ref{tmMaxBubDecom}]
Denote by $S$ the set of $\sim$ equivalence classes.  Pick a component $A\subset\tilde{\Sigma}\setminus\partial\Sigma$. For each $c\in S$ whose elements lie in $A$ or which is conjugation invariant assign an element $I_c\in c$ of maximal modulus. For any other $c$ define $I_{c}:=\overline{I}_{\overline{c}}$. Let now
\[
\tilde{\mathcal{B}}=\{I_c|c\in S\}.
\]
In the exceptional case where $genus(\tilde{\Sigma})=1$ and $Mod(I_0)\geq L_1,$ we add $I_0$ to $\tilde{\mathcal{B}}$.  It is follows from Lemma~\ref{lmcConjInv} that $\tilde{\mathcal{B}}$ is conjugation invariant. By Lemma \ref{lmEsdDisjInEq} the elements of $\tilde{\mathcal{B}}$ are pairwise essentially disjoint.
Now let
\[
\mathcal{B}=\{C(2K_1,2K_1;I)|I\in \tilde{\mathcal{B}}\}.
\]
We show that $\mathcal{B}$ is a maximal $\mu$-decomposition. We check the stability condition. Let $v\in V_{\mathcal{B}}$. We distinguish between the following cases.
\begin{enumerate}
\item $2genus(\Sigma_v)+|\pi_0(\partial\Sigma_v)|\geq 3$. In this case stability is automatic.
\item $genus(\Sigma_v)=1$ and $\partial\Sigma_v=\emptyset$. Stability is a consequence of Assumption~\ref{as1}.
\item
$genus(\Sigma_v)=0$ and $|\pi_0(\partial\Sigma_v)|=2$. Then if $\Sigma_v$ is bordered by two inequivalent elements of $LN$, this case is covered by Lemma \ref{lmBubDecompmustab}. Otherwise, we must have that $genus(\tilde{\Sigma})=1,$ and either $\Sigma_v$ is the complement of a long neck, or $\Sigma_v$ is bordered by $I_0$ and some element $I_1$ of $LN$. In the first case, stability follows by definition of long necks. We treat the second case. Write $I=\Sigma_v$ and suppose by contradiction that $\mu(I)<\delta_2/6$.  Let
\[
M=I_0\cup I_1\cup I\cup\overline{I_1\cup I},
\]
and let $N=C(L_0,L_0;M).$ Suppose first $\partial\Sigma=\emptyset.$ Then $Mod(N)>Mod(I_0)$ and $\mu(N)\leq \delta_2/6$. This contradicts the choice of $I_0$. Suppose now $\partial\Sigma\neq\emptyset$. Then $I_0\subset N$ and $M\setminus N\subset I_1\cup \overline{I}_1$. By assumption $\mu(M)<\delta_1,$ so
\[
\mu(I_0)\leq\mu(N)\leq\delta_2/9.
\]
This again contradicts the choice of $I_0$.
\item
$|\pi_0(\partial\Sigma_v)|=1$. In this case $\Sigma_v$ is a component of the complement of a long neck and so the claim follows by definition.
\end{enumerate}

It remains to check the maximality condition. Suppose by contradiction that $\Sigma_v$ contains a long neck $I'$. Then there is a $c\in S$ such that $I'\in c$. Clearly,
 \[
 C(K_1,K_1;I')\cap C(K_1,K_1;I_c)=\emptyset.
 \]
That is, $I'$ and $I_c$ are essentially disjoint. Since $I_c$ has maximal modulus in $c$, this contradicts Lemma \ref{lmEsdDisjInEq}.

\end{proof}

\section{$(\mu,h)$-adaptedness}\label{SEcMuAdapt}

\begin{tm}\label{tmBubDecEst}
Let $\mathcal{F}$ be a uniformly thick thin family. Let $(\Sigma,\mu)\in \mathcal{F}$ satisfy Assumption~\ref{as1}.
\begin{enumerate}
\item
If $\partial\Sigma=\emptyset,$ there is a conformal constant curvature metric $h$ on $\Sigma$ and a $(\mu,h)$-adapted bubble decomposition $\mathcal{B}$ of $\Sigma$ with constants independent of $(\Sigma,\mu)$.
\item
If $\partial\Sigma\neq\emptyset,$ there is a conjugation invariant conformal constant curvature metric $h$ on $\Sigma_\C$ and a conjugation invariant $(\mu,h)$- adapted bubble decomposition $\mathcal{B}$ of $\Sigma_\C$ with constants independent of $(\Sigma,\mu)$.
\end{enumerate}
\end{tm}

For the rest of this section fix a $(\Sigma,\mu)\in\mathcal{F}$ satisfying Assumption~\ref{as1} with the understanding that all constants depend only on $\mathcal{F}$ and not on the particular $(\Sigma,\mu)$ we chose. Let $\mathcal{B}$ be a maximal $\mu$-decomposition as in Theorem \ref{tmMaxBubDecom}.  To prove that $\mathcal{B}$ satisfies the estimates in part~\ref{dfBubDecEst13} of Definition~\ref{tmBubDecEst}, we need to introduce some notation.

Associate to $\mathcal{B}$ a graph $G_{\mathcal{B}}$ as follows. As the vertex set of $G_{\mathcal{B}}$ take $V_{\mathcal{B}}.$ Add an outgoing half edge $l$ from $v$ for each element of $\pi_0(\partial\Sigma_v)$. For any $v\in V_{\mathcal{B}},$ denote by $\mathcal{H}_v$ the set of half edges going out of $v$. For $l\in\mathcal{H}_v$ denote by $\gamma_l$ the boundary component corresponding to $l$. Half edges $l_1$ and $l_2$ are connected to one another in $G_\mathcal{B}$ if and only if there is an element of $I\in\mathcal{B}$ such that $\partial I=\gamma_{l_1}\cup\gamma_{l_2}$. There is thus a two to one correspondence between half edges and elements of $\mathcal{B}$. For $l\in\mathcal{H}_v$, write $I_l$ for the corresponding element of $\mathcal{B}.$

Let $v\in V_{\mathcal{B}}$. An external boundary component of $\Sigma_v$ is an element $\gamma\in\pi_0(\partial\Sigma_v)$ such that $\gamma$ is either not contractible in $\tilde{\Sigma}$ or satisfies
\[
Diam(\gamma;h)=Diam(\Sigma_v;h)\footnote{Note that if $\Sigma_v$ is formed by removing any number of small discs from sphere, then $E_v=\emptyset$.}.
\]
Let $E_v\subset \mathcal{H}_v$ denote the half edges corresponding to the external boundary components of $\Sigma_v$, and let
\[
F_v:=\mathcal{H}_v\setminus E_v.
\]
For each $l\in F_v$, $\gamma_l$ is the boundary of a disc $B_l\subset\tilde{\Sigma}$. Write
\[
Cl(\Sigma_v):=\Sigma_v\bigcup_{\{l\in F_v\}}B_l \subset \Sigma_\C.
\]

\begin{lm}\label{lmNoLoN}
There is a constant $f_1$ with the following significance. Let $v\in V_{\mathcal{B}}$ and let $I\subset Cl(\Sigma_v)$ be a neck. In case $I=B\setminus B'$ for discs $B'\subset B\subset\Sigma$, suppose that $\partial B'\subset\Sigma_v$. Let
\[
n(I):=\left|\{l\in F_v:\gamma_l\cap I\neq \emptyset\}\right|.
\]
Then $Mod(I)\leq f_1(\mu(I\cap \Sigma_v)+n(I)+1)$.
\end{lm}
\begin{proof}
Let $L=Mod(I)$.  For any integer $0\leq i< \lfloor {L/L_1}\rfloor$ let
\[
I_i:= S(iL_1,(i+1)L_1;I).
\]
Since $\Sigma_v$ contains no long necks, we must have
\[
\mu(I_i)>\delta_2/6.
\]
Let $S_1$ denote the set of those  $0\leq i< \lfloor {L/L_1}\rfloor$ that satisfy $I_i\subset\Sigma_v$ and let $S_2$ be the rest. Clearly,
\[
\frac{L}{L_1}\leq |S_1|+|S_2|+1,
\]
and
\[
|S_1|\leq \frac{6\mu(I\cap \Sigma_v)}{\delta_2}.
\]
To complete the proof we need to bound $|S_2|$.

If $i\in S_2$, there is an $l\in F_v$ such that $I_i\cap B_l\neq \emptyset$. We show that there as at most one $j\neq i$ such that  $B_l$ meets $I_j$. For this, let $J_0\in \mathcal{B}$ be the unique element such that
\[
\gamma_l\subset\partial J_0.
\]
There is a $J_1\in LN$ such that $J_0=C(2K_1,2K_1;J_1)$. Let
\[
J=S(Mod(J_1)-3K_1,Mod(J_1)-K_1;J_1).
\]
Suppose now by contradiction that $B_l$ meets three successive sub-cylinders $I_{i-1},$ $I_{i}$ and $I_{i+1}$. By the assumption of the lemma, $B_l$ does not contain any of the $I_i$. Therefore, $\gamma_l\cap C(K_1,K_1;I_{i})\neq\emptyset$. But $\gamma_l\subset C(K_1,K_1;J)$. So, $J$ and $I_i$ are not essentially disjoint.

On the other hand, we show that the fact that $B_l\not\subset I_i$ implies that $J$ and $I_i$ are essentially disjoint. Let $k:=b_1(I_i\cup J)$.  By Theorem \ref{tmEssDisj} it suffices to show that $k>1$. Suppose by contradiction that $k\leq 1$. If $k=0$ then $I$ is trivial and its interior disc $B$ is contained in $J$. But since $I$ is a neck, $\mu(B)\geq\delta_1/2$ whereas $\mu(J)\leq\mu(J_1)\leq\delta_2/6$. Suppose now that $k=1$. Since we are assuming $B_l\not\subset I_i$, this is only possible if $I$ is trivial and $B_l\cap B\neq\emptyset$. But then by the assumption of the Lemma we have that $B_l\subset B$, in contradiction to $I_i\cap B\neq\emptyset.$ We conclude that $I_i$ and $J$ are essentially disjoint. The contradiction shows that $B_l$ meets at most two sub-cylinders. We thus conclude that
\[
|S_2|\leq 2n(I).
\]

\end{proof}
For any $v\in V_\mathcal{B}$ let $n_v=|F_v|$ and $\mu_v=\mu(\Sigma_v)$.
\begin{lm}\label{GammaEsts}
There are constants $f_i$, for $i=2,...,9$, with the following significance. Let $v\in V_{\mathcal{B}}$ and let $l\in E_v$.
\begin{enumerate}
\item\label{GammaEsts1}
\[
\ell(\gamma_l;h_v)\geq f_2e^{-f_3(\mu_v+n_v)}.
\]
\item\label{GammaEsts2}
For all $x\in\Sigma_v$
\[
\inj(\Sigma_v, x;h_v)\geq f_4e^{-f_5(\mu_v+n_v)}.
\]
\item\label{GammaEsts3}
Let $l'\neq l\in \mathcal{H}_v$. Then
\[
d(\gamma_l,\gamma_{l'};h_v)\geq f_6e^{-f_7(\mu_v+n_v)}.
\]
\end{enumerate}

\end{lm}
\begin{rem}\label{remInjBoundary}
$\inj(\Sigma_v, x;h_v)$ is defined as the supremum of all $r$ such that any unit speed geodesic ray
\[
\alpha:\left[0,\min\left\{r,d_{h_v}(p,\partial\Sigma_v)\right\}\right]\to\Sigma_v
\]
emanating from $p$ minimizes length.
\end{rem}
\begin{proof}
Let let $g=genus(\tilde{\Sigma})$.  Let
\[
m(v):=2genus(Cl(\Sigma_v))+|E_v|.
\]
We distinguish between various possibilities for $m(v)$ and $g$.
\begin{enumerate}
\item
$m(v)=0$. In this case $E_v=\emptyset$, so only part~\ref{GammaEsts2} is not vacuous. But part~\ref{GammaEsts2} is obvious.
\item
$m(v)=1$ and $g=0$. By carefully inspecting the definition of external boundary parts \ref{GammaEsts1} and ~\ref{GammaEsts2}are seen to hold. We show part~\ref{GammaEsts3}. By construction, there are $J,J'\in LN$ such that
\[
I_l=C(2K_1,2K_1;J),
\]
and
\[
I_{l'}=C(2K_1,2K_1;J').
\]
Let $N$ be the component of
\[
C(K_1,K_1;J)\setminus I_l,
\]
for which $\gamma_l\subset\partial N$. $N$ is a tubular neighborhood of $\gamma_l$. By essential disjointness of $J_l$ and $J_{l'}$ we have that $N\cap \gamma_{l'}=\emptyset$. We have
\[
Mod(N)=K_1.
\]
Denote by $r$ the metric width of $N$. That is, the distance between the two boundary components. Then
\[
K_1=\frac1{s_v}\int_0^r\frac{dx}{h_{\theta,FS}}\leq \int_0^r\frac{dx}{h_{\theta,FS}},
\]
where $h_{\theta,FS}$ is Fubini Study metric in appropriate coordinates. Take $f_6$ to be the solution of
\[
K_1=\int_0^{f_6}\frac{dx}{h_{\theta,FS}}.
\]
$f_7$ may be taken to vanish.

\item
$m(v)=1$ and $g>0.$ This case is similar to the previous case.
\item
$m(v)=2$ and $g=0.$ In this case it can be verified that $Cl(\Sigma_v)=B\setminus B'$ for two concentric discs in $\Sigma_\C$. Suppose first that $Cl(\Sigma_v)$ is contained in a hemisphere. Then the only additional thing to address after the case $m(v)=1$ is to estimate $\ell(\partial B';h_v)$. Applying Lemma~\ref{lmNoLoN} to $Cl(\Sigma_v)$ we have
\[
Mod(cl(\Sigma_v))\leq  f_1(\mu(I\cap \Sigma_v)+n(I)+1).
\]
On the other hand we denote by $r$ and $r'$ the radii of $B$ and $B'$ with respect to $h_v$, then
\[
\log(r/r')\leq cMod(I),
\]
for an appropriate constant. Now note that $r=1$, so the claim follows. If $Cl(\Sigma_v)$ is not contained in a hemisphere, cut $Cl(\Sigma_v)$ in two along a concentric equator and repeat the same argument.
\item
$m(v)=2$ and $g=1$. Only part~\ref{GammaEsts2} is not vacuous. But $d_v$ in this case is proportional to the modulus of $\Sigma_v$ which is appropriately bounded by Lemma \ref{lmNoLoN}.
\item
$m(v)= 2$ and $g>1$. Let $e\in E_v$. Then there is a simple closed geodesic $\gamma$ such that $\gamma_e\subset\mathcal{C}(\gamma)$. Write $I=Cl(\Sigma_v)$ and let $\gamma_1$ and $\gamma_0$ be the components of $\partial I$. It is easy to see that $I$ is a sub-cylinder of $\mathcal{C}(\gamma)$. Therefore, $\gamma_0$ and $\gamma_1$ have constant $\rho$ coordinates $x_0$ and $x_1$, respectively. For $r\in[x_0,x_1]$ let
\[
\gamma_r:=\{z\in I|\rho(z)=r\},
\]
and let $r_{\min}\in [x_0,x_1]$ be the point where $\ell(\gamma_r)$ obtains its minimum, $\ell_{\min}$. Without loss of generality, assume $|x_0|\leq |x_1|$. We have
\begin{align}\label{EqRatio}
\ln\frac{\ell_{\min}}{\ell_1}&=\ln\frac{h_{\theta}(r_{\min})}{h_{\theta}(x_1)}\\
&\geq-\int_{x_0}^{x_1}\frac{h_{\theta}'(x)}{h_{\theta}(x)}dx\notag\\
&\geq -\int_{x_0}^{x_1}\frac{1}{\pi h_{\theta}(x)}dx\notag\\
&=-\frac1{\pi}ModI\notag\\
&\geq -\frac1{\pi}f_1(\mu_v+n_v).\notag
\end{align}
Here we rely on the inequality ${h_{\theta}'(x)}=\ell(\gamma)\sinh x/(2\pi)\leq 1/\pi$ for $x\in w(\gamma)$. On the other hand,
 \begin{equation}\label{diamRatio}
 \frac{d_v}{{\ell(\gamma_1;h)}}\leq \frac{|x_1-x_0|+\ell(\gamma_0;h)}{\ell(\gamma_1;h)}.
\end{equation}
 But
\begin{equation}\label{diamModEst}
|x_1-x_0|\leq \frac{\ell(\gamma_1;h)}{2\pi}\int_{x_0}^{x_1}\frac{dx}{h_{\theta}(x)}= \frac{\ell(\gamma_1;h)}{2\pi}Mod(I),
\end{equation}
where for the inequality we relied on the equation
\[
h_{\theta}(x)=\frac1{2\pi}\ell(\{\rho=x\};h)\leq \ell_1.
\]
Combining estimates~\eqref{EqRatio},~\eqref{diamRatio} and~\eqref{diamModEst}, we obtain
\[
\frac{\ell_{min}}{d_v}\geq\frac{exp\left(-\frac1{\pi}f_1(\mu_v+n_v)\right)}{Mod(I)+1}.
\]
Together with Lemma~\ref{lmNoLoN}, this implies part~\ref{GammaEsts1}

Part~\ref{GammaEsts2} is a consequence of Eq. \eqref{injEq} as follows. For any $p\in I$, let $x=\rho(p)$ and $d=w(\gamma)-|x|$. We have,
\begin{align}
\inj(p;\Sigma,h)&=\sinh^{-1}(\cosh\frac1{2}\ell(\gamma)\cosh d-\sinh d)\\
&=\sinh^{-1}(e^{-d}+(\cosh\frac1{2}\ell(\gamma)-1)\cosh d)\notag\\
&\geq \sinh^{-1}(e^{-d})\notag\\
&=\ln(e^{-d}+\sqrt{e^{-2d}+1})\notag\\
&=e^{-d}+o(e^{-d}).\notag
\end{align}
Let $\xi= |x_1|-|x|\in[0,|x_1|-|x_0|]$. We have
\begin{align}
\inj(p;\Sigma,h_{v})&\geq c\frac{e^{-d}}{\frac{d_v}{\ell_1}\ell_1}\notag \\
&\geq c\frac{e^{-d}}{\ell(\gamma)\cosh x_1 (Mod(I)+1)}\notag\\
&\geq c\frac{e^{-w(\gamma)}}{\ell(\gamma)(Mod(I)+1)}e^{-\xi}.\notag
\end{align}
It is straightforward to verify that there is lower bound on the expression $\frac{e^{-w(\gamma)}}{\ell(\gamma)}$ which is independent of $\gamma$. Since $\xi\leq |x_1|-|x_0|<diam(I;h)$, the claim follows.

Given the estimate on $\ell(\gamma_i)$ the proof of part~\ref{GammaEsts3} in the current case is similar to that of the case $m(v)=1$ and $g=0$. We omit the details.
\item
$m(v)>$ and $g=0$. Considering the definition of external boundary components, there is no such case.
\item
$m(v)>2$ and $g>1$. Decompose
\[
\Sigma_v =(Thick(\Sigma;h)\cap \Sigma_v)\cup (Thin(\Sigma;h)\cap \Sigma_v).
\]
The components of $(Thin(\Sigma;h)\cap \Sigma_v)$ behave exactly as the case $m(v)=2$ and $g=1$ and contain all the external boundary components. It remains to estimate on $\inj$ and $(Thick(\Sigma;h)\cap \Sigma_v)$, but this is a tautology.
\end{enumerate}
\end{proof}
To establish the rest of the estimates in Definition~\ref{dfBubDecEst}, we introduce some further notation. For any $v\in V_{\mathcal{B}}$ let
\[
r_v:=\begin{cases}
\frac13\min_{z\in {Cl(\Sigma_v)}}\inj(\Sigma_\C,z;h),&genus(\Sigma_\C)>0,\\
\min\left\{\sqrt{\frac{\delta_1}{2\pi K_0}},\frac{\pi}{3}\right\},&genus(\Sigma_\C)=0.
\end{cases}
\]
Let $B=B_r(p;h_v)\subset Cl(\Sigma_v)$ be a clean geodesic disc. Define
\[
r_B:=\begin{cases}
\min\{s_vr_v,d(p,\partial Cl(\Sigma_v);h_v)\},& p\in\partial\Sigma,\\
\min\{s_vr_v,d(p,\partial Cl(\Sigma_v);h_v),\frac1{2} d(p,\overline{p};h_v)\},&p\not\in\partial\Sigma.
\end{cases}
\]

\begin{lm}\label{argammaest}
There is a constant $f_2$ with the following significance. Let $B=B_r(p;h_v)\subset Cl(\Sigma_v)$ be a clean disc of radius $r$ satisfying
\[
\mu(B)\geq\delta_1/2.
\]
Then
\[
r\geq \frac15r_Be^{-f_2(\mu_v+n_v)}.
\]
\end{lm}
\begin{proof}
Let
\[
I=A(r_B,r, p).
\]
If $r_B< 5r$ we are done, so suppose $r_B\geq 5r$. It follows that $I\in\mathcal{A}_h$. Also, $I\subset Cl(\Sigma_v).$ We claim that $I$ is a neck.  For this we need to verify that both $B$ and $\Sigma':=\Sigma_{\C}\backslash B_{r_B}(p;h_v)$ are stable. But $B$ is stable by assumption. In the case where
\[
genus(\Sigma_{\C})>0,
\]
$\Sigma'$ is immediately seen to be stable. When $genus(\Sigma)=0$, stability of $\Sigma'$ follows from the fact that $h$ satisfies the condition of Lemma \ref{LmSpSphMet}. Furthermore, $I$ satisfies the condition of Lemma \ref{lmNoLoN}. So,
\[
f_1\{\mu(Cl(I))+n(Cl(I))+1\}\geq Mod(Cl(I))>c\log\frac{r_B}{2r},
\]
for an appropriate constant $c$. This inequality gives the claim.
\end{proof}
\begin{lm}\label{InGammaEst}
There are constants $f_i$, $10\leq i\leq 13$, with the following significance. Let $B=B_r(p;h_v)\subset Cl(\Sigma_v)$ be a clean disc such that $\mu(B)\geq\delta_1/2$. Suppose
\begin{equation}\label{ineqdpbd}
d(p,\partial Cl(\Sigma_v);h_v)>f_6e^{-f_7(\mu_v+n_v)}.
\end{equation}
 Then
\[
r\geq f_{10}e^{-f_{11}(\mu_v+n_v)}.
\]
\end{lm}

\begin{proof}
By Lemma \ref{GammaEsts} we have that
\[
\inj(x;h_v)\geq f_4e^{-f_5(\mu_v+n_v)}
\]
for all $x\in Cl(\Sigma_v)$. So, by assumption \eqref{ineqdpbd}, when $B$ is conjugation invariant we have
\[
r_B\geq \min\left\{f_4e^{-f_5(\mu_v+n_v)},\frac1{2}f_6e^{-f_7(\mu_v+n_v)},s_v\sqrt{\frac{\delta_1}{2\pi K_0}},s_v\frac{\pi}{3}\right\}.
\]
the claim now follows by Lemma \ref{argammaest}.

To prove the claim for any clean $B$ we need to further bound
\[
r':=\frac1{2}d(p,\overline{p};h_v)
\]
from below by an exponent in $\mu_v+n_v$. In fact, to prove the Lemma, it suffices to estimate $r+r'$ by such an exponent. We may suppose
\begin{equation}\label{DummyAssum}
r+r'<\frac1{2}f_6e^{-f_7(\mu_v+n_v)},
\end{equation}
for otherwise we are done. Let $p'$ be the midpoint of the shortest geodesic segment connecting $p$ with $\overline{p}$. Combining inequality \eqref{DummyAssum} with inequality \eqref{ineqdpbd} we get
\[
d(p',\partial Cl(\Sigma_v)) >\frac1{2}f_6e^{-f_7(\mu_v+n_v)}.
\]
Let $B':=B_{r+r'}(p')$. Then $\subset Cl(\Sigma_v)$. Furthermore,
\[
r_{B'}>\frac1{4}f_6e^{-f_7(\mu_v+n_v)},
\]
By Lemma \ref{argammaest} this implies
\[
r'+r\geq \frac1{20} f_6e^{-(f_7+f_2)(\mu_v+n_v)}.
\]
\end{proof}
\begin{cy}\label{CyInGammaEst}
\begin{enumerate}
\item
For any $l\in F_v$
\[
\ell(\Sigma_v;h_v)\geq  f_{10}e^{-f_{11}(\mu_v+n_v)}.
\]
\item
For any $l_1,l_2\in F_v$ we have
\[
d(\gamma_{l_1},\gamma_{l_2};h_v)\geq f_{12}e^{-f_{13}(\mu_v+n_v)}.
\]
\end{enumerate}
\end{cy}
\begin{proof}
\begin{enumerate}
\item By Lemma \ref{GammaEsts}\ref{GammaEsts3}, the assumptions of Lemma \ref{InGammaEst} hold in particular for $B=B_l$ where $l\in F_v$.
\item This follows by the same proof as that of Lemma \ref{GammaEsts}\ref{GammaEsts3}.
\end{enumerate}
\end{proof}
In the following, for any $\gamma\in \pi_0(\partial\Sigma),$ let $N_{\gamma}:=B_{f_{12}e^{-f_{13}(\mu_v+n_v)}}(\gamma;h_v)$. Without loss of generality we assume $f_{12}\leq f_6$ and $f_{13}\geq f_7$.
\begin{cy}\label{cyCFEst}
For any $\gamma\in\pi_0(\partial\Sigma_v)$
\begin{align*}
\frac{d\nu_{h_v}}{d\nu_{h_{st}}}\Big|_{N_{\gamma}}\geq f_{10}e^{-f_{11}(\mu_v+n_v)}.
\end{align*}
\end{cy}
\begin{proof}
Using cylindrical coordinates on $N_{\gamma}$ let
\[
\gamma_r =\{z\in N_{\gamma}|\rho(z)=r\}.
\]
We have
\[
\frac{d\nu_{h_v}}{d\nu_{h_{st}}}(r,\theta)=\frac1{2\pi}\ell(\gamma_r).
\]
If $\gamma\in F_v,$ Lemmas \ref{InGammaEst} and \ref{GammaEsts}\ref{GammaEsts3} imply
\[
\ell(\gamma_r)\geq f_{10}e^{-f_{11}(\mu_v+n_v)}.
\]
Otherwise, this is just Lemma \ref{GammaEsts}\ref{GammaEsts1}.
\end{proof}
\begin{lm}\label{cyDerEst}
There are constants $f_{14},f_{15},$ such that for any $p\in\Sigma_v$,
\[
  \frac{d\mu}{d\nu_{h_{v}}}(p)\leq f_{15}e^{f_{14}(\mu_v+n_v)}.
\]
\end{lm}
\begin{proof}
 Let $p\in\Sigma_v$ be the point where the supremum of $\frac{d\mu}{d_{\nu_{h_v}}}$ is obtained and let $d$ be its value. Let  $c=f_6e^{-f_7(\mu_v+n_v)}$. If $d(p,\partial Cl(\Sigma_v);h_v)< c$ then, by construction of $\mathcal{B}$, there is a long neck $I$ so that $p$ is contained in $C(\pi+c_2,\pi+c_2;I)\subset C(K_1,K_1;I)$. Thus, by Lemma \ref{ExpCylDEst}, $\frac{d\mu}{d\nu_{h_{st}}}\leq a$. The claim now follows from Corollary \ref{cyCFEst}.

Otherwise, if $d<1/c$ we are done. If $d>1/c$, consider the disc $B=B_{\frac1{d}}(p;h_v)$. Then $\mu(B)>\delta_1$ by Remark~\ref{rmGradApp} and $B\subset\Sigma_v$, so the bound follows immediately from Lemma~\ref{InGammaEst}.
\end{proof}

\begin{proof}[Proof of Theorem~\ref{tmBubDecEst}]
Let $\mathcal{B}$ be a maximal $\mu$-decomposition as in Theorem \ref{tmMaxBubDecom}. Note that this $\mathcal{B}$ satisfies part~\ref{dfBubDecEst11} of Definition~\ref{dfBubDecEst}. Indeed, for any $I$ in $\mathcal{B}$ there is an $I'\in LN$ such that $I=C(K_1,K_1;I')$. But we assumed in Definition \ref{dfEsDisj} that $K_1\geq c_2+\pi$. By definition of $LN$, $\mu(I')<\delta_2$. The claim now follows from Lemma \ref{ExpCylDEst}. That $\mathcal{B}$ satisfies part~\ref{dfBubDecEst12} of Definition~\ref{dfBubDecEst} is just Definition~\ref{defMaxBubDecom}\ref{tmMaxBubDecom2}.
 The estimates of part~\ref{dfBubDecEst}\ref{dfBubDecEst13} are the content of Lemmas  \ref{GammaEsts} and \ref{InGammaEst}, Corollary~\ref{CyInGammaEst}, and Lemma \ref{cyDerEst}.
\end{proof}

\section{Proof of Theorems~\ref{tmBubDecEst1},~\ref{tmBubDecEst2}, and~\ref{tmBubDecEst3}}

\begin{proof}[Proof of Theorems~\ref{tmBubDecEst1},~\ref{tmBubDecEst2}, and~\ref{tmBubDecEst3}]
Let
\[
\mathcal{M}=\{(\Sigma,\mu_u)|(\Sigma,u)\in\mathcal{F}).
\]
According to Theorem 2.8 in \cite{GS13}, the hypotheses of Theorems~\ref{tmBubDecEst1},~\ref{tmBubDecEst2}, and~\ref{tmBubDecEst3} imply that $\mathcal{M}$ is uniformly thick thin. If $(\Sigma,u)\in\mathcal{M}$ satisfies Assumption~\ref{as1} the theorems follow from Theorem~\ref{tmBubDecEst}. Otherwise, let $\mathcal{B}=\emptyset$. If $genus(\Sigma_\C)=0$ we must have a metric $h$ satisfying condition~\ref{LmSpSphMet0} in Lemma \ref{LmSpSphMet}. Stability follows from the fact that $u$ is non-constant and the rest of the claims are obvious. Now assume $genus(\Sigma_\C)=1$. All parts of the theorem hold vacantly except for stability, the derivative estimate and the injectivity radius estimate. Stability follows from the monotonicity inequality as follows. The injectivity radius of $M$ is uniformly bounded away from zero by a constant $r$. Let $p\in(\Sigma)$. Since $u$ represents a nontrivial homology class $u(\Sigma)\not\subset B_r(p;g_J)$. By the boundedness of the curvature and by the monotonicity inequality,
\[
 E(\Sigma_\C;u)>Area(u(\Sigma)\cap B_r(p;g_J))\geq cr^2
 \]
for a constant $c>0$.

To bound the injectivity radius and derivative we need to bound
\[
Diam(\Sigma_\C;h).
\]
For this, it suffices to bound the modulus of $\Sigma_\C$. For any $x>0$, let
\[
L:=c_2+\pi+\frac{\ln\{a(c_2+\pi+x)\}}{c_3},
\]
where the constants are as in  Lemma~\ref{ExpCylDEst}. If $Mod I>2L$,  any point $p\in\Sigma_\C$ is at the center of a cylinder of modulus $2L$. Lemma~\ref{ExpCylDEst} then implies that
\[
\frac{d\mu}{d\nu_{h_{st}}}(p)\leq \frac1{(c_2+\pi+x)}\mu(\Sigma_\C).
\]
Pick $x$ large enough so that
\[
\frac{4\pi L}{c_2+\pi+x}<1.
\]
We then have the contradiction
\[
\mu(\Sigma_\C)\leq 2L\sup_{p\in\Sigma_\C}\frac{d\mu}{d\nu_h}(p)\leq \frac{4\pi L}{c_2+\pi+x}\mu(\Sigma_\C)<\mu(\Sigma_\C).
\]
The derivative estimate is an immediate consequence of Remark~\ref{rmGradApp} and the global bound $\mu(\Sigma_\C)<\delta_2<\delta_1$. The radius on injectivity of $h_v$ is just the inverse of the diameter multiplied by a suitable constant.
\end{proof}

\bibliographystyle{amsabbrvc}
\bibliography{RefTCY2}

\vspace{.5 cm}
\noindent
Institute of Mathematics \\
Hebrew University, Givat Ram \\
Jerusalem, 91904, Israel \\
\end{document}